\documentclass[10pt,oneside,a4paper,table]{article}

\usepackage{amscd,amssymb,amsmath,amsthm,mathrsfs,dsfont}
\usepackage[a4paper]{geometry}
\usepackage{tikz} 
\usetikzlibrary{positioning}
\usetikzlibrary{decorations.pathreplacing}
\usetikzlibrary{backgrounds}
\usetikzlibrary{patterns}
\usetikzlibrary{calc} 
\usepackage[shortlabels]{enumitem}
\usepackage{comment}

\usepackage{caption}
\usepackage{array}
\usepackage{subcaption}
\usepackage{graphicx}
\usepackage[thinc]{esdiff}
 
\usepackage{hyperref}
\hypersetup{
    colorlinks=true,
    linkcolor=blue,
    filecolor=black,      
    urlcolor=blue,
    citecolor=blue,
}

\usepackage{mathtools}
\usepackage{mathrsfs}
\usepackage{lettrine}
\usepackage{mfirstuc}
\usepackage{multicol}
\usepackage{multirow}
\usepackage{makecell}
\usepackage{dirtytalk}

\newtheorem{theorem}{Theorem}
\newtheorem{proposition}{Proposition}
\newtheorem{lemma}{Lemma}
\newtheorem{corollary}{Corollary}
\newtheorem{remark}{Remark}
\newtheorem{definition}{Definition}
\def\N{\mathbb{N}}

\def\R{\mathbb{R}}

\DeclareUnicodeCharacter{2212}{-}
\numberwithin{equation}{section}
\usepackage{fancyhdr}
\usepackage[backend=bibtex,style=ieee,giveninits=true,sorting=anyt]{biblatex}
\newcommand\shorttitle{Dynamical GnG on trees}

\begin{document}

\hypersetup{linkcolor=black}
\hypersetup{urlcolor=black}
\title{Dynamical Gibbs-non-Gibbs transitions for finite-alphabet models on trees}
\author{Sebastian Bergmann\footnotemark[1] \and
Christof Külske\footnotemark[2] \and Niklas Schubert\footnotemark[3]}
\date{August 14, 2026}

\maketitle
\begin{abstract}
We study finite-alphabet spin models on trees evolving under independent symmetric spin-flip dynamics. On the lattice the time-evolved plus measure of the low temperature 
Ising model in zero external field was shown to be non-quasilocal for all sufficiently large times in \cite{EnFeHoRe02}.
On the tree however the time-evolved plus phase of the Ising model 
behaves differently, as the quasilocal Gibbs property is first lost, but then recovered at a later time \cite{EnErIaKu12}. 
We first extend this result by showing that large-time reentry into the quasilocal Gibbs regime on the tree happens more generally for all finite alphabet models, when the dynamics is started in \textit{$f$-stable} Gibbs states. 
These are defined in terms of a time-independent sharp condition on the 
 homogeneous recursion $f$ describing the model, 
which can be checked explicitly.

As our second and opposite result, we develop a method for proving large-time persistence of non-quasilocality of time-evolved measures.  We show that even all configurations can become bad at large times, when the initial Gibbs measure corresponds to an $f$-saddle and give an explicit illustration 
for the Potts model. 

In our proofs we study the influence of spatially inhomogeneous perturbations to the solutions of a time-dependent fixed point problem. 
\end{abstract}
\textbf{Keywords:} Gibbs measures, non-Gibbsian measures, quasilocality, 
Dynamical Gibbs-non-Gibbs transition, 
trees, time-evolution, Potts model. \newline
\noindent\textbf{MSC2020 subject classifications:} 82B26 (primary);
60K35 (secondary).
\footnotetext[1]{Ruhr-University Bochum, Germany}
\footnotetext[0]{E-Mail: \href{mailto:Sebastian.Bergmann@ruhr-uni-bochum.de}{Sebastian.Bergmann@ruhr-uni-bochum.de}}
\footnotetext[0]{\url{https://www.researchgate.net/profile/Sebastian-Bergmann}}
\footnotetext[1]{E-Mail: \href{mailto:Christof.Kuelske@ruhr-uni-bochum.de}{Christof.Kuelske@ruhr-uni-bochum.de}; ORCID iD: \href{https://orcid.org/0000-0001-9975-8329}{0000-0001-9975-8329}}
\footnotetext[0]{\url{https://math.ruhr-uni-bochum.de/en/faculty/professorships/stochastics/group-kuelske/staff/christof-kuelske/}}
\footnotetext[2]{E-Mail: \href{mailto:Niklas.Schubert@ruhr-uni-bochum.de}{Niklas.Schubert@ruhr-uni-bochum.de}; ORCID iD: \href{https://orcid.org/0009-0000-8912-4701}{0009-0000-8912-4701}}
\footnotetext[0]{\url{https://math.ruhr-uni-bochum.de/en/faculty/professorships/stochastics/group-kuelske/staff/niklas-schubert/}}
\tableofcontents
\hypersetup{linkcolor=blue}

\section{Introduction}
The behavior of infinite-volume spin models under stochastic time-evolutions of a spin-flip type has a long and rich history. Relevant issues which are studied are well-definedness of interacting  dynamics 
in infinite volume \cite{Li85}, and the characterization of measures which are invariant 
under time evolution and their relation to Gibbs measures, in particular when the dynamics is reversible \cite{HoSt77}. 
One next studies non-equilibrium questions where, depending on the type of dynamics one looks 
at scaling limits \cite{KiLa99} and the relation to PDEs, and one is interested in 
attractor properties of a non-scaled infinite-volume lattice system itself. 
This is particularly interesting 
when the dynamics corresponds to low temperature, in other words 
the set of invariant measures is not a singleton, and the dynamics is started away from the invariant states
\cite{JaKo25, JaKu19}.  

In our present work we are interested in the 
characterization of trajectories of time-evolved measures 
$t\mapsto \mu_t$,  and their regularity or loss of regularity, 
in the spirit of the questions proposed and studied in the seminal paper \cite{EnFeHoRe02}. 
In that paper the authors found there is the 
possibility of the appearance of \say{non-Gibbsian} 
(more precisely \textit{non-quasilocally} Gibbsian) states 
along some parts of the trajectory $\mu_t$.
The loss of the quasilocal Gibbs property means that conditional probabilities of the system show some long-range dependence, see Definition \ref{def: Goodness of measures} and Remark \ref{rk: Goodness implies Gibbs}. If such  changes of regularity happen at certain transition times, one speaks of a "dynamical Gibbs-non-Gibbs transition".
More precisely, the authors of 
\cite{EnFeHoRe02} proved that the quasilocal Gibbs property does not hold for large enough times for the time-evolved lattice Ising model, 
if the low-temperature plus state in zero external field 
on the lattice is subjected to site-wise independent or weakly dependent spin flips. 
On the other hand, for the lattice Ising model started with a measure in non-zero external magnetic 
field, after an initial loss of Gibbsianness, there is a recovery of the quasilocal Gibbs property. 
Studies for time-evolved lattice or mean-field models with various local states spaces were performed by \cite{KuRe06,KuLe07,EnRu09,EnFeHoRe10,ReRoWi10,FeHoMa13}.

Now, changing to infinite tree graphs as  base spaces, it is a general principle that probabilistic models on trees often show different and in some sense even richer behavior than their counterparts on lattices \cite{LyPe16}.  An example for this in our context is the behavior of the time-evolved Ising model 
in zero field on the regular tree. For the problem of regularity of trajectories of Gibbs measures under dynamics  the following was found: 
The low-temperature plus phase 
under independent spin-flip dynamics first loses the  quasilocal Gibbs property, but at a larger  
time recovers it again (much in contrast to the lattice behavior). Moreover such a re-entry into Gibbs notably does not occur for the time-evolved free state on the tree (constructed with open boundary conditions).  
This highlights that on trees there may be a variety of qualitatively different behavior of trajectories of measures 
for the same system parameters, where the differences are crucially depending on the Gibbs measure in which the model is initialized \cite{EnErIaKu12, BeKiKu23}. 
The contribution of the present paper is to give this a closer and unifying look, in the broader framework of general nearest-neighbor spin models 
with finite local spin space (including the Potts model). 
Let us now describe the models, the time evolution and the results informally.

{\bf Nearest-neighbor tree models.} Consider a nearest-neighbor spin 
model with $q$ possible spin values on a regular tree with $d+1$ 
neighbors, and recall the following useful background we will need.   It is well-known that  
all automorphism-invariant extremal phases correspond to fixed points $l$ of a model dependent map on the simplex of length-$q$ probability vectors, that is, to equations  
of the form $l=f(l)$. Here $f$, given in \eqref{eq: Homogeneous time-independent FP equation} is a non-linear self-map on the simplex, which depends on model parameters.  
Typically in regimes corresponding 
to low temperatures or strong coupling it admits different fixed points, corresponding to different infinite-volume states (or phases). 
The finite-volume marginals of the infinite-volume measure $\mu^l$ which corresponds to a given fixed point solution $l$ have an  explicit expression
in terms of $l$, see \eqref{eq: Finite-volume marginals time-independent GM}. Any such measure (be it extremal or not) also has a representation as a tree-indexed Markov chain 
(or splitting Gibbs measure), with an $l$-dependent transition matrix. 
Conversely, every extremal Gibbs measure 
is representable as a tree-indexed Markov chain, see \cite{Ge11,Z83}.

The free (or open boundary) 
state of an Ising or Potts model becomes non-extremal at sufficiently large values 
of its coupling constant, different from the phase transition value for uniqueness of the Gibbs measure
\cite{BlRuZa95,Io96,Sl11}. 
For properties of the extremal decomposition of free states of such tree models in strong coupling regimes, which tends to spread over uncountably many inhomogenous pure states and is atomless (or "glassy"), see \cite{GaMaRuSh20, JiLa26, CoKuLe26}. Extensions of this phenomenon to the more general class of $A$-localized states are discussed in \cite{AbHeKuMa24, KuSc25a}.  

{\bf Time evolution: Symmetric spin flip between $q$ spin values.}
Our time evolution is simply described as follows: At each site of the tree, there is a 
Poissonian clock (with standard exponential waiting times), which is given independently over the sites. When the clock rings, the spin at the site takes a value different from its current state, and these different values are chosen with equal probability. The corresponding transition kernel is presented  in \eqref{eq: time-dependent spin flip}. 
The paper covers two complementary possibilities for the long-time behavior, 
depending on the starting measure in the following two findings. 
\\

\textit{{\bf Result 1 (Long-time recovery of quasilocal Gibbs property).}}  In Theorem \ref{thm: Goodness at large times}, we show that \textit{$f$-stability} of a fixed point $l$ satisfying $l=f(l)$ on the simplex is related to quasilocal Gibbsianness (Definition \ref{def: Goodness of measures}) for the time-evolved measure $\mu_t^l$ given in \eqref{eq: Definition time evolved measure}.  
More precisely, we say that $l$ is $f$-stable if the differential of $f$ at $l$ as 
a linear map on the (tangent space of) the simplex has spectral norm (maximal eigenvalue in modulus) strictly smaller than one. 
We stress that this condition does not 
refer to the time evolution, and that it is a strictly local condition for $f$ at $l$ on the simplex.  
It does not exclude the existence of multiple fixed points 
(that is phase transitions in the time-zero model). Our 
general result concerning long-time recovery is then the following:

\textit{Under $f$-stability of $l$, 
there is a (possibly large but) 
finite time $t_{R}$ such that 
for all $t\geq t_{R}$ the time-evolved 
measure $\mu_t^{l}$ is a quasilocal Gibbs measure.}

Moreover the influence of boundary-condition variations to conditional probabilities of  $\mu_t^{l}$, is exponentially suppressed in the distance on the tree, see \eqref{eq: exponential decay of b.c. influence on boundary laws}. 
We believe that the condition is optimal when we exclude the threshold cases with neutral eigenvalues of $Df(l)$ having modulus one.

Are large times really needed for such long-time regularity  to hold? 
Yes. Still assuming $f$-stability of $l$, the time-evolved measure may very well be non-quasilocally Gibbs at \textit{intermediate times} in low temperature regimes when several fixed points of $f$ occur.  This is proved 
for the particular example of the plus state of the Ising model in \cite{EnErIaKu12}.  For an analogous statement in the  discrete Widom-Rowlinson model, see \cite{BeKiKu23}, where due to the difference to our model caused by the hardcore nature of that model, proof techniques based on subtree percolation were used. 
To analyze non-quasilocality at intermediate times is therefore a separate issue, which seems to need model-dependent treatment. \\

{\bf Result 2 (Long-time loss of quasilocal Gibbs property where all configurations are bad). }
For fixed points of the time-independent equation $l=f(l)$ which have \textit{some} 
unstable directions (which may be \textit{$f$-saddles} or even \textit{repelling fixed points} under $f$)
generically the following happens: \textit{All configurations} will be points of discontinuity (bad configurations) of the conditional 
probabilities of $\mu_t^l$ (Definition \ref{def: Goodness of measures}), when $t$ is sufficiently large but finite.  
A first example where this was proved to happen is the \textit{free state} of the low-temperature Ising model \cite{EnErIaKu12}. 
In this work we outline a general mechanism that may apply to a broad class of finite-alphabet models to prove this statement, based on deformations of basins of attraction of the $f$-stable fixed points. These deformations are caused by variations 
of the spins appearing as conditionings in the time-evolved layer of the model. The proof is carried out in detail for $f$-saddles and repelling fixed points of the three-state ($q=3$) Potts model in the small temperature regions, see Theorem \ref{thm: Loss without recovery}. For a temperature-dependent overview of the long-time results for the three-state Potts model, see Table \ref{tab: long-time results time-evolved measure Potts}.\\

{\bf Novelties. }
We show that long-time recovery of quasilocality always holds for arbitrary finite-alphabet models in the full region of $f$-stability. 
We also investigate the opposite statement, that the non-quasilocal Gibbs property occurs at arbitrary large times where even all configurations are bad, when the evolution starts in a Gibbs measure corresponding to an $f$-saddle or a repelling fixed point. We prove that this occurs in a particular multidimensional example.  We expect the mechanism to apply more broadly, where one unstable direction of a fixed point of $f$ may be enough.

\section{Model and main results}
In this section, we introduce the finite-state spin models on Cayley trees
considered throughout the paper. We first recall the description of spatially
homogeneous tree-indexed Markov chains in terms of transfer operators and
boundary laws. We then define the independent spin-flip dynamics and state
our main large-time results: recovery of quasilocal Gibbsianness for $f$-stable boundary laws and loss of quasilocal Gibbsianness for $f$-saddles and repelling fixed points.

\subsection{Gibbs measures on trees and time-evolution}

We first fix the graph-theoretic notation for finite-state spin models on Cayley trees and then introduce the class of spatially homogeneous tree-indexed Markov chains described by boundary laws.

\indent\textbf{Cayley tree graphs.} Let $(V,E)$ be a graph, where $V$ denotes the set of vertices and $E\subset \big\{\{x,y\}:~x,y\in V,~x\neq y\big\}$ the set of undirected edges. If $\{x,y\}\in E$, then the vertices $x$ and $y$ are nearest neighbors and we write $x\sim y$. We write $\langle x,y\rangle$ for the directed edge from $x$ to $y$
whenever $\{x,y\}\in E$. We consider Cayley tree graphs of order $d$, in which every vertex has $d+1$ nearest neighbors, and distinguish a vertex $\rho\in V$ at the root. Given a subset $\Lambda \subset V$, we define the outer boundary $\partial\Lambda:=\{y\in \Lambda^c:~\{x,y\}\in E,~x\in \Lambda\}$ as the set of vertices in $\Lambda^c$ which have a nearest neighbor in $\Lambda$. Further, we define $\overline{\Lambda}:=\Lambda \cup \partial \Lambda$ as the closure of $\Lambda$ and write $\Lambda \Subset V$ for a finite subset of $V$. The distance $d(x,y)$ between two vertices $x,y\in V$ is defined by the length of the unique shortest path from $x$ to $y$.

\textbf{Finite spin models.} In this paper, we consider spin models on $V$ with a finite
state space $S=\{1,\dots,q\}$ where $q\in \N_{\geq 2}$. Therefore, the configuration space is given by $\Omega:=S^V$ and is endowed with the product $\sigma$-algebra $\mathcal{F}:=\mathcal{P}(S)^{\otimes V}$. For a subset $\Lambda \subset V$, we denote by $\Omega_\Lambda:=\{\omega_\Lambda=(\omega_x)_{x\in \Lambda}:~\omega_x\in S\}$ the restriction of the configuration space on $\Lambda$. Furthermore, we define the projection $\sigma_\Lambda:\Omega \rightarrow \Omega_\Lambda$ as the map $\sigma_\Lambda(\omega)=\omega_\Lambda$. By a slight abuse of notation, we sometimes write $x$ instead of $\{x\}$. The $\sigma$-algebra generated by the projections $(\sigma_x)_{x\in \Lambda}$ is denoted by $\mathcal{F}_\Lambda$. If $\Lambda$ and $\Delta$ are disjoint subsets of $V$ and $\omega,\eta\in \Omega$, we define the concatenation $\omega_\Lambda\eta_\Delta\in \Omega_{\Lambda\cup \Delta}$ of these two configurations as the configuration $\xi_{\Lambda \cup \Delta}\in \Omega_{\Lambda \cup \Delta}$ satisfying $\xi_x=\omega_x$ for all $x\in \Lambda$ and $\xi_y=\eta_y$ for all $y\in \Delta$. A function $\varphi:\Omega\rightarrow \R$ is called local if it is $\mathcal{F}_\Lambda$-measurable for a subset $\Lambda\Subset V$ and $\Lambda$ is called the support of $\varphi$.

\textbf{Tree-indexed Markov chains.}
The model interaction is described by a transfer operator
$Q:S\times S \rightarrow (0,\infty)$ which is assumed to be symmetric, i.e., it satisfies $Q(i,j)=Q(j,i)$ for all $i,j\in S$. We discuss spatially homogeneous tree-indexed Markov chains
$\mu^l\in \mathcal{M}_1(\Omega,\mathcal{F})$ specified by a transfer
operator $Q$ and a spatially homogeneous boundary law $l$, whose
finite-volume probabilities are given by
\begin{equation}\label{eq: Finite-volume marginals time-independent GM}
    \mu^l(\sigma_{\Lambda \cup \partial\Lambda}
    =\omega_{\Lambda \cup \partial\Lambda})
    =\frac{1}{Z_\Lambda}
    \prod_{\substack{\{x,y\}\in E\\ \{x,y\}\cap \Lambda \neq \emptyset}}
    Q(\omega_x,\omega_y)
    \prod_{z\in \partial\Lambda}l(\omega_z)
\end{equation}
where $\Lambda \Subset V$ is a connected set, $\omega\in \Omega$, and $Z_\Lambda$ is a suitable normalizing constant.
Here, $l\in (0,\infty)^S$ is characterized by the spatially homogeneous
boundary law equation associated with $Q$, i.e., there exists a constant
$c>0$ such that
\begin{equation}\label{eq: Definition boundary law equation}
    l(j)=c\Big(\sum_{i\in S}Q(j,i)l(i)\Big)^d
\end{equation}
for all $j\in S$. Since boundary laws are determined only up to multiplication by a positive
constant, we choose the normalized representative $l\in\Delta^{q-1}$, where 
\begin{equation*}
    \Delta^{q-1}
:= \left\{ v=(v_1,\dots,v_q)\in[0,1]^q :
\sum_{i=1}^q v_i=1 \right\}
\end{equation*}
denotes the $(q-1)$-simplex in $\R^q$. Under this convention, the
boundary law equation \eqref{eq: Definition boundary law equation} becomes
the fixed point equation $f(l)=l$, where
$f:\Delta^{q-1}\rightarrow \Delta^{q-1}$ is defined by
\begin{equation}\label{eq: Homogeneous time-independent FP equation}
    f(l):=\frac{(Ql)^{\odot d}}{\|(Ql)^{\odot d}\|_1}.
\end{equation}
Here, the expression $v^{\odot d}:=(v_1^d,\dots,v_q^d)$ denotes the $d$-fold
Hadamard power of $v=(v_1,\dots,v_q)\in \R^q$, and $\|v\|_1:=\sum_{i=1}^q|v_i|$ denotes the $\ell^1$-norm of $v$. For more details about tree-indexed Markov chains, see \cite{Ge11,Z83,Ro13}.

\textbf{Independent spin-flip dynamics.} Let $\mu^l\in \mathcal{M}_1(\Omega,\mathcal{F})$ be a tree-indexed Markov chain. 
We define its time evolution by the action of a Markov semigroup $(\pi_t)_{t\geq 0}$ on $\Omega$, i.e., $\mu^l_t := \mu^l \pi_t$ for all $t\geq 0$.
In this paper, we consider site-wise independent spin-flip dynamics. 
The corresponding single-site transition kernel is denoted by $P_t:S\times S\to[0,1]$,
and is assumed to be independent over the sites. The quantity $P_t(i,j)$ denotes the probability that a spin initially in state $i$ is found in state $j$ at time $t$ and it is explicitly given by
\begin{equation}\label{eq: time-dependent spin flip}
    \quad P_t(i,j):=\frac{1}{q}+\left(\delta_{ij}-\frac{1}{q}\right)e^{-\frac{q}{q-1}t}
\end{equation}
for $i,j\in S$. The full transition kernel $\pi_t$ on $\Omega$ is the product kernel $\pi_t(\sigma,d\eta)
    \allowbreak= \allowbreak
    \bigotimes_{x\in V} P_t(\sigma_x,d\eta_x)$.
Therefore, we have
\begin{equation}\label{eq: Definition time evolved measure}
    \mu^l_t(\varphi)
=
\int_{\Omega}
\sum_{\eta_\Lambda\in\Omega_\Lambda}
\varphi(\eta_\Lambda)
\prod_{x\in\Lambda}P_t(\omega_x,\eta_x)
\mu^l(d\omega)
\end{equation}
for local functions $\varphi$ with support $\Lambda \Subset V$. In this work, we use a two-layer representation of the time-evolved model.
We write $\omega$ for the initial configuration at time zero (first-layer configuration) and $\eta$ for the
configuration at time $t$ (second-layer configuration). The canonical coordinate maps on the configuration space are
still denoted by $\sigma_x$. In particular, under $\mu^l_t$, the random variable $\sigma_x$ denotes
the spin at site $x$ at time $t$.

\subsection{Long-time recovery under time evolution for $f$-stable initial states}

We begin by identifying time-evolved measures that are quasilocally Gibbs in the following sense.

\begin{definition}\label{def: Goodness of measures}
A configuration $\eta \in \Omega$ is \underline{good} for $\mu^l_t$ if for all $x\in V$, all local
functions $\varphi$, and each sequence of finite volumes $\Lambda_n\uparrow V$,
\begin{equation}\label{eq: Goodness of measures definition}
    \lim_{n\rightarrow \infty}\sup_{\substack{\xi^1,\xi^2\in \Omega\\ \Delta:~\Lambda_n\subset \Delta \Subset V}}\Big|\mu^l_t(\varphi|\eta_{\Lambda_n\setminus \{x\}}\xi^1_{\Delta \setminus \Lambda_n})-\mu^l_t(\varphi|\eta_{\Lambda_n\setminus \{x\}}\xi^2_{\Delta \setminus \Lambda_n})\Big|=0.
\end{equation}
A configuration which is not good is called \underline{bad}. If all configurations are good for $\mu^l_t$, the measure $\mu^l_t$ is called \underline{quasilocally Gibbs}.
\end{definition}

The following theorem shows that $f$-stability of the underlying boundary law implies quasilocal Gibbsianness of the time-evolved measure at all sufficiently large times.

\begin{theorem}[Long-time recovery]\label{thm: Goodness at large times}
Let $Q:S\times S\rightarrow (0,\infty)$ be a transfer operator which is symmetric and homogeneous on all edges, 
and let $l$ be a boundary law corresponding to an $f$-stable fixed point of \eqref{eq: Homogeneous time-independent FP equation}, where $f$-stable means that the spectral radius of $Df(l)$ is strictly smaller than one. Let $\mu^l_t$ be the 
time-evolved measure defined in \eqref{eq: Definition time evolved measure}, obtained from the tree-indexed Markov chain $\mu^l$ under the site-wise independent spin-flip dynamics. Then there exists 
$t_R< \infty$ such that $\mu^l_t$ is quasilocally Gibbs for all $t\geq t_R$.
\end{theorem}

The proof is presented in Section \ref{sec: Proof of long time recovery} and is based on the decay of boundary dependence under suitable families of tree contractions. Our Definition \ref{def: Goodness of measures} of quasilocal Gibbsianness is indeed associated with the existence of a quasilocal specification, see Remark \ref{rk: Goodness implies Gibbs}.

\subsection{Loss without recovery in time-evolved Potts: Total badness caused by an $f$-saddle}

This subsection provides a concrete example demonstrating that, in contrast to Theorem \ref{thm: Goodness at large times}, the presence of a single unstable direction at the fixed point associated with a boundary law may imply total badness of the corresponding measure. More precisely, this unstable direction corresponds to an eigenvalue of the differential of the boundary-law recursion whose modulus is greater than $1$. By \textit{total badness}, we mean that all configurations are bad in the sense of Definition \ref{def: Goodness of measures}.
To illustrate the mechanism, we consider the three-state ferromagnetic Potts model on the Cayley tree with three nearest neighbors per vertex. The spin space is $S=\{1,2,3\}$, and the transfer operator $Q=(Q(i,j))_{i,j\in S}$ is given by $Q(i,j)=\exp(\beta \delta_{ij}),$
where $\beta>0$ denotes the inverse temperature.

\begingroup

\begin{table}[ht]
\scriptsize
    \centering
    \begin{tabular}{c||c|c}
    \rowcolor{gray!50}

        \multicolumn{1}{c||}{\vtop{\hbox{\thead{}}}}& \vtop{\hbox{\thead{$\theta\in(4,\infty)$}}}  &  \vtop{\hbox{\thead{$\theta\in(1+2\sqrt{2},4)$}}} \\\hline\hline
        \thead{$l_{\text{free}}$}& \thead{total badness\\ (repelling FP)\\ (Theorem \ref{thm: Loss without recovery})} &  \thead{recovery\\ ($f$-stable)\\
        (Theorem \ref{thm: Goodness at large times})}  \\ \hline
        \rule{0pt}{6mm}
  {\thead{$l_{-}$\\ ($f$-saddle)}} & \thead{total badness\\ (Theorem \ref{thm: Loss without recovery})} & {\thead{total badness\\ selection mechanism\\ changes in example\\ Figure \ref{fig: Discrete trajectories normalized recursion small theta} }}\\ \hline
   {\thead{$l_+$\\ ($f$-stable)}}& \thead{recovery\\ (Theorem \ref{thm: Goodness at large times})}  &  \thead{recovery\\ (Theorem \ref{thm: Goodness at large times})}

    \end{tabular}\medskip
    \begin{tikzpicture}[overlay, remember picture]
\begin{scope}[shift={(-10.5,0.5)}, scale=1]

\coordinate (A) at (0,0);
\coordinate (B) at (1,0);
\coordinate (C) at (0.5,{sqrt(3)/2});

\draw[thick] (A) -- (B) -- (C) -- cycle;

\pgfmathsetmacro{\thetaVal}{5}
\pgfmathsetmacro{\disc}{(\thetaVal-1)^2 - 8}
\pgfmathsetmacro{\xminus}{((\thetaVal-1 - sqrt(\disc))/2)^2}
\pgfmathsetmacro{\xplus}{((\thetaVal-1 + sqrt(\disc))/2)^2}


\coordinate (Lp1) at (0.25,0.15);
\coordinate (Lp2) at (0.75,0.15);
\coordinate (Lp3) at (0.5,0.6);

\coordinate (Lm1) at (barycentric cs:A=\xminus,B=1,C=1);
\coordinate (Lm2) at (barycentric cs:A=1,B=\xminus,C=1);
\coordinate (Lm3) at (barycentric cs:A=1,B=1,C=\xminus);

\coordinate (Lf) at (barycentric cs:A=1,B=1,C=1);

\fill[red] (Lp1) circle (0.02);
\fill[red] (Lp2) circle (0.02);
\fill[red] (Lp3) circle (0.02);

\fill[black] (Lm1) circle (0.020);
\fill[black] (Lm2) circle (0.020);
\fill[black] (Lm3) circle (0.020);

\fill[black] (Lf) circle (0.045);
\end{scope}

\begin{scope}[shift={(-8.7,0.5)}, scale=1]

\coordinate (A) at (0,0);
\coordinate (B) at (1,0);
\coordinate (C) at (0.5,{sqrt(3)/2});

\draw[thick] (A) -- (B) -- (C) -- cycle;

\pgfmathsetmacro{\thetaVal}{3.83}
\pgfmathsetmacro{\disc}{(\thetaVal-1)^2 - 8}
\pgfmathsetmacro{\xminus}{((\thetaVal-1 - sqrt(\disc))/2)^2}
\pgfmathsetmacro{\xplus}{((\thetaVal-1 + sqrt(\disc))/2)^2}


\coordinate (Lp1) at (0.25,0.15);
\coordinate (Lp2) at (0.75,0.15);
\coordinate (Lp3) at (0.5,0.6);

\coordinate (Lm1) at (barycentric cs:A=\xminus,B=1,C=1);
\coordinate (Lm2) at (barycentric cs:A=1,B=\xminus,C=1);
\coordinate (Lm3) at (barycentric cs:A=1,B=1,C=\xminus);

\coordinate (Lf) at (barycentric cs:A=1,B=1,C=1);

\fill[red] (Lp1) circle (0.02);
\fill[red] (Lp2) circle (0.02);
\fill[red] (Lp3) circle (0.02);

\fill[black] (Lm1) circle (0.02);
\fill[black] (Lm2) circle (0.02);
\fill[black] (Lm3) circle (0.02);

\fill[black] (Lf) circle (0.045);
\end{scope}

\begin{scope}[shift={(-10.5,-0.85)}, scale=1]

\coordinate (A) at (0,0);
\coordinate (B) at (1,0);
\coordinate (C) at (0.5,{sqrt(3)/2});

\draw[thick] (A) -- (B) -- (C) -- cycle;

\pgfmathsetmacro{\thetaVal}{5}
\pgfmathsetmacro{\disc}{(\thetaVal-1)^2 - 8}
\pgfmathsetmacro{\xminus}{((\thetaVal-1 - sqrt(\disc))/2)^2}
\pgfmathsetmacro{\xplus}{((\thetaVal-1 + sqrt(\disc))/2)^2}


\coordinate (Lp1) at (0.25,0.15);
\coordinate (Lp2) at (0.75,0.15);
\coordinate (Lp3) at (0.5,0.6);

\coordinate (Lm1) at (barycentric cs:A=\xminus,B=1,C=1);
\coordinate (Lm2) at (barycentric cs:A=1,B=\xminus,C=1);
\coordinate (Lm3) at (barycentric cs:A=1,B=1,C=\xminus);

\coordinate (Lf) at (barycentric cs:A=1,B=1,C=1);

\fill[red] (Lp1) circle (0.02);
\fill[red] (Lp2) circle (0.02);
\fill[red] (Lp3) circle (0.02);

\fill[black] (Lm1) circle (0.045);
\fill[black] (Lm2) circle (0.045);
\fill[black] (Lm3) circle (0.045);

\fill[black] (Lf) circle (0.02);
\end{scope}

\begin{scope}[shift={(-8.7,-0.85)}, scale=1]

\coordinate (A) at (0,0);
\coordinate (B) at (1,0);
\coordinate (C) at (0.5,{sqrt(3)/2});

\draw[thick] (A) -- (B) -- (C) -- cycle;

\pgfmathsetmacro{\thetaVal}{3.83}
\pgfmathsetmacro{\disc}{(\thetaVal-1)^2 - 8}
\pgfmathsetmacro{\xminus}{((\thetaVal-1 - sqrt(\disc))/2)^2}
\pgfmathsetmacro{\xplus}{((\thetaVal-1 + sqrt(\disc))/2)^2}


\coordinate (Lp1) at (0.25,0.15);
\coordinate (Lp2) at (0.75,0.15);
\coordinate (Lp3) at (0.5,0.6);

\coordinate (Lm1) at (barycentric cs:A=\xminus,B=1,C=1);
\coordinate (Lm2) at (barycentric cs:A=1,B=\xminus,C=1);
\coordinate (Lm3) at (barycentric cs:A=1,B=1,C=\xminus);

\coordinate (Lf) at (barycentric cs:A=1,B=1,C=1);

\fill[red] (Lp1) circle (0.02);
\fill[red] (Lp2) circle (0.02);
\fill[red] (Lp3) circle (0.02);

\fill[black] (Lm1) circle (0.045);
\fill[black] (Lm2) circle (0.045);
\fill[black] (Lm3) circle (0.045);

\fill[black] (Lf) circle (0.02);
\end{scope}

\begin{scope}[shift={(-10.5,-2.15)}, scale=1]

\coordinate (A) at (0,0);
\coordinate (B) at (1,0);
\coordinate (C) at (0.5,{sqrt(3)/2});

\draw[thick] (A) -- (B) -- (C) -- cycle;

\pgfmathsetmacro{\thetaVal}{5}
\pgfmathsetmacro{\disc}{(\thetaVal-1)^2 - 8}
\pgfmathsetmacro{\xminus}{((\thetaVal-1 - sqrt(\disc))/2)^2}
\pgfmathsetmacro{\xplus}{((\thetaVal-1 + sqrt(\disc))/2)^2}


\coordinate (Lp1) at (0.25,0.15);
\coordinate (Lp2) at (0.75,0.15);
\coordinate (Lp3) at (0.5,0.6);

\coordinate (Lm1) at (barycentric cs:A=\xminus,B=1,C=1);
\coordinate (Lm2) at (barycentric cs:A=1,B=\xminus,C=1);
\coordinate (Lm3) at (barycentric cs:A=1,B=1,C=\xminus);

\coordinate (Lf) at (barycentric cs:A=1,B=1,C=1);

\fill[red] (Lp1) circle (0.045);
\fill[red] (Lp2) circle (0.045);
\fill[red] (Lp3) circle (0.045);

\fill[black] (Lm1) circle (0.02);
\fill[black] (Lm2) circle (0.02);
\fill[black] (Lm3) circle (0.02);

\fill[black] (Lf) circle (0.02);

\end{scope}

\begin{scope}[shift={(-8.7,-2.15)}, scale=1]

\coordinate (A) at (0,0);
\coordinate (B) at (1,0);
\coordinate (C) at (0.5,{sqrt(3)/2});

\draw[thick] (A) -- (B) -- (C) -- cycle;

\pgfmathsetmacro{\thetaVal}{3.83}
\pgfmathsetmacro{\disc}{(\thetaVal-1)^2 - 8}
\pgfmathsetmacro{\xminus}{((\thetaVal-1 - sqrt(\disc))/2)^2}
\pgfmathsetmacro{\xplus}{((\thetaVal-1 + sqrt(\disc))/2)^2}


\coordinate (Lp1) at (0.25,0.15);
\coordinate (Lp2) at (0.75,0.15);
\coordinate (Lp3) at (0.5,0.6);

\coordinate (Lm1) at (barycentric cs:A=\xminus,B=1,C=1);
\coordinate (Lm2) at (barycentric cs:A=1,B=\xminus,C=1);
\coordinate (Lm3) at (barycentric cs:A=1,B=1,C=\xminus);

\coordinate (Lf) at (barycentric cs:A=1,B=1,C=1);

\fill[red] (Lp1) circle (0.045);
\fill[red] (Lp2) circle (0.045);
\fill[red] (Lp3) circle (0.045);

\fill[black] (Lm1) circle (0.02);
\fill[black] (Lm2) circle (0.02);
\fill[black] (Lm3) circle (0.02);

\fill[black] (Lf) circle (0.02);
\end{scope}

\end{tikzpicture}
    \caption{Long-time behavior of time-evolved measures $\mu_t^l$ associated with fixed points $l$ of the homogeneous normalized boundary-law recursion $f$ for the three-state Potts model, classified in Lemma \ref{lem: Fixed points normalized simplex Potts model}. The results are discussed throughout the non-uniqueness regime $\theta=\text{exp}(\beta)>1+2\sqrt{2}$, excluding the critical value $\theta=4$, which marks the stability threshold of the free fixed point. Additionally, the locations of the corresponding class of fixed points, depending on $\theta$, on the simplex are depicted schematically on the left side of the table. Note that, as $\theta$ decreases and passes $4$, the $f$-saddles cross the free fixed point and move from the seperatrix into different dominant chambers, defined in \eqref{eq: p_i-dominant chamber}.}
    \label{tab: long-time results time-evolved measure Potts}
\end{table}   
\endgroup

We write $\theta:=\text{exp}(\beta)>1$ and let $l=(l_1,l_2,l_3)$ be a spatially homogeneous boundary law. The homogeneous boundary-law equation, introduced in \eqref{eq: Definition boundary law equation}, is given by
\begin{equation*}
    l_i =
    c\left(
    \sum_{j=1}^3 Q(i,j)l_j
    \right)^2=c\left((\theta-1)l_i+\sum_{i=1}^3 l_i\right)^2,
\end{equation*}
for $i=1,2,3$ and for some $c>0$.
We pass to normalized coordinates on the simplex $\Delta^2$.
That is, for a positive boundary law $l=(l_1,l_2,l_3)$, we define $p_i=\frac{l_i}{l_1+l_2+l_3}$ for $i\in \{1,2,3\}$.
Since boundary laws that differ only by a multiplication by a positive constant induce the same Gibbs measure, this normalization merely selects the representative on the simplex of the corresponding equivalence class. The normalized homogeneous boundary-law recursion $f:\Delta^2\rightarrow \Delta^2$ introduced in \eqref{eq: Homogeneous time-independent FP equation} has components
\begin{equation}\label{eq: Normalized homogeneous BL recursion Potts}
f_i(p):=
\frac{\left(1+(\theta-1)p_i\right)^2}
{\sum_{k=1}^3 \left(1+(\theta-1)p_k\right)^2}
\end{equation}
for all $i\in \{1,2,3\}$ where $p=(p_1,p_2,p_3)\in \Delta^2$.
Thus, the spatially homogeneous normalized boundary laws are represented by the fixed points of $f$, that is, by solutions $p\in\Delta^2$ of $p=f(p)$.

\begin{lemma}[Classification of all fixed points]\label{lem: Fixed points normalized simplex Potts model}
Assume that $\theta>1+2\sqrt{2}$ and define 
\begin{equation*}
    x_\pm:=
\left(
\frac{\theta-1\pm\sqrt{(\theta-1)^2-8}}{2}
\right)^2.
\end{equation*}
Furthermore, for each $i\in \{1,2,3\}$, define the vectors $l^{(i)}_-,l^{(i)}_+\in \Delta^2$ componentwise by
\begin{equation}\label{eq: Def symmetry breaking FP's Potts}
    l^{(i)}_\pm(j):=\frac{1+(x_\pm-1)\delta_{ij}}{x_\pm+2}
\end{equation}
for $j\in \{1,2,3\}$.
Then, for $\theta\neq 4$, the points $l^{(i)}_-$ where $i\in \{1,2,3\}$ are precisely the fixed points of $f$ with exactly one unstable direction, that is, the fixed points $l$ for which $Df(l)$ has exactly one eigenvalue of modulus strictly greater than $1$. We refer to these fixed points as $f$-saddles. The remaining symmetry-breaking fixed points are $l^{(i)}_+$ where $i\in \{1,2,3\}$, and they are all $f$-stable.
Moreover, both eigenvalues of $Df$ at the free fixed point $l_{\text{free}}:=\left(\frac{1}{3},\frac{1}{3},\frac{1}{3}\right)$ have modulus strictly greater than $1$ if $\theta>4$ and strictly smaller than $1$ if $\theta<4$.
\end{lemma}

Discrete trajectories under the time-independent homogeneous boundary-law recursion $f$ on the simplex, together with the corresponding fixed points, are presented in the left panel of Figure \ref{fig: Discrete trajectories normalized recursion}. The proof of Lemma \ref{lem: Fixed points normalized simplex Potts model} is presented in \hyperref[Appendix: Computation BL's and stability]{Appendix C}.

\begin{theorem}[Loss without recovery]\label{thm: Loss without recovery}
Let $\theta>4$, and let $l$ be either an $f$-saddle or the free fixed point. By permutation symmetry, it suffices to consider $l=l_-^{(1)}$ or $l=l_{\text{free}}$, as classified in Lemma \ref{lem: Fixed points normalized simplex Potts model}.

Then there exists a time $t_{B}<\infty$ such that, for every $t\geq t_B$, there exists $\epsilon(t)>0$ with the following property: For every cofinal sequence $(\Lambda_n)_{n\in \N}$, one can choose a sequence of finite sets
$(\Delta_n)_{n\in \N}$ satisfying $\Lambda_n\subset \Delta_n \Subset V$ for every $n\in \N$,
such that 
\begin{equation}\label{eq: Root marginal separation thm}
\liminf_{n\rightarrow \infty}\left|\mu^{l}_t\left(\eta_o=2\mid\eta_{\Lambda_n \setminus \{o\}},2_{\Delta_n \setminus \Lambda_n}\right)
-\mu^{l}_t\left(\eta_o=2\mid\eta_{\Lambda_n \setminus \{o\}},3_{\Delta_n \setminus \Lambda_n}\right)\right|\geq \epsilon(t)
\end{equation}
for every $\eta\in \{1,2,3\}^V$.  
\end{theorem}
In particular, at large enough times all configurations are bad, which is the strongest possible form of a pathology which can occur for a non-Gibbsian measure. The corresponding statements for the other $f$-saddles follows after a suitable permutation of the spin values and the associated boundary conditions.  The proof is presented in Section \ref{sec: Proof of total badness}. The long-time results for the time-evolved measures of the three-state Potts model are summarized in Table \ref{tab: long-time results time-evolved measure Potts}. Note that the free state 
undergoes a transition of its long-time behavior, as a function of the inverse temperature, between low and intermediate temperatures. This phenomenon does not occur for the free state of the Ising model \cite{EAEVIGK12}.

We expect analogous results to hold for large enough $\theta$ and initial states of a Potts model for general $q$, 
corresponding to boundary laws which are of the form $l_-=(a,\dots,a,b,\dots,b)$ where precisely $m \in \{1, \dots, q -1\}$ components 
are equal to $a$, and $a> b$. These are (up to permutation of the indices) the only solutions to $l=f(l)$ for the homogeneous boundary law recursion corresponding to the $q$-state Potts model, see \cite{KuRo14}, and they exist 
at $\theta$ sufficiently large, depending on $q, m, d$.

\section{Proof of Theorem \ref{thm: Goodness at large times} (long-time recovery) via \\ tree-contraction property}\label{sec: Proof of long time recovery}

We derive a spatially recursive description of the finite-volume conditional probabilities of $\mu^l_t$, see \eqref{eq: Root marginals recursive description}. 
This recursion is based on a 
two-layer representation for the system, where the variables of the two layers are given by the spins at time zero and at time $t$, respectively. While we are interested 
in the behavior at time $t$, the recursions to be considered are acting 
on the first layer of time-zero variables, but depend 
as a conditioning on general spatially inhomogeneous choices of the second-layer variables at time $t$. This allows us to reformulate quasilocal Gibbsianness (Definition \ref{def: Goodness of measures}) as a decay-of-boundary-influence property for suitably defined boundary messages, see  \eqref{eq: Definition time evolved BL}. These messages are indexed by directed edges and encode the information transmitted across an edge from the boundary condition outside of a given finite connected subset to the inside.  This propagation of boundary information is governed by the recursions introduced below, which are summarized in Table \ref{tab: overview recursions}. 
Key issue of the proof is to see that at large times 
the influence of the second-layer conditionings becomes controllable, as 
we show that they are tied to the behavior of the map $f$ in the neighborhood of $f$-stable boundary laws.

\subsection*{Recursive description of the conditional probability at the root via boundary messages}

 Let $\Lambda\Subset V$ be connected with $o\in \Lambda$ and $\eta\in \Omega$. Using the two-layer representation of the time-evolved measure, where
$\omega$ denotes the configuration at time zero and $\eta$ the
configuration at time $t$, we obtain
\begin{equation*}
     \mu^l_t(\sigma_o=\eta_o \mid \sigma_{\Lambda\setminus\{o\}}=\eta_{\Lambda\setminus\{o\}})=\frac{1}{Z_\Lambda(\eta_{\Lambda\setminus\{o\}})}\sum_{\omega_\Lambda \in \Omega_\Lambda}\mu^l(\sigma_\Lambda=\omega_\Lambda)\prod_{x\in \Lambda}P_t(\omega_x,\eta_x)
\end{equation*}
where $Z_\Lambda(\eta_{\Lambda\setminus\{o\}}):=\sum_{\omega_\Lambda\in \Omega_\Lambda}\mu^l(\sigma_\Lambda=\omega_\Lambda)\prod_{x\in \Lambda\setminus \{o\}}P_t(\omega_x,\eta_x)$.
Since the initial measure $\mu^l$ is a tree-indexed Markov chain, it admits a representation in terms of the transfer operator $Q$ and the boundary law $l$. Thus the above expression can be written as
\begin{equation}\label{eq: Time evolved cond prob. repr. through BL}
    \mu^l_t(\sigma_o=\eta_o|\sigma_{\Lambda\setminus \{o\}}=\eta_{\Lambda\setminus\{o\}})=\frac{1}{Z_\Lambda(\eta_{\Lambda\setminus\{o\}})}\sum_{\omega_{\overline{\Lambda}} \in \Omega_{\overline{\Lambda}}}\prod_{\substack{\{x,y\}\in E\\ \{x,y\}\cap \Lambda\neq\emptyset}} Q(\omega_x,\omega_y)\prod_{z\in {\partial \Lambda}} l(\omega_z)\prod_{w\in \Lambda}P_t(\omega_w,\eta_w).
\end{equation}
The \textit{time-independent inhomogeneous normalized recursion}, derived from the boundary-law equation in \eqref{eq: Definition boundary law equation}, describes how the contributions from the $d$ descendant subtrees are recursively combined in the time-zero model. In simplex coordinates, this recursion is given by the map $F:\big(\Delta^{q-1}\big)^d\longrightarrow \Delta^{q-1}$,
defined by
\begin{equation}\label{eq: Inhomogeneous time-independent BL equation}
F(v^{(1)},\dots,v^{(d)})
:=\frac{Qv^{(1)}\odot\cdots\odot Qv^{(d)}}{\left\|Qv^{(1)}\odot\cdots\odot Qv^{(d)}\right\|_1}.
\end{equation}
Here $\odot$ denotes componentwise multiplication, that is, for $v,w\in \R^q$, one has
$(v_1,\dots,v_q)\odot(w_1,\dots,w_q):=(v_1w_1,\dots,v_qw_q)$.
The spatially homogeneous boundary law recursion $f$ in \eqref{eq: Homogeneous time-independent FP equation} is recovered from $F$ through the relation $f(v)=F(v,\dots,v)$ for all $v\in \Delta^{q-1}$. We call $(v^{(1)},\dots,v^{(d)})$ \textit{messages}. Under the recursion, these messages are propagated from the outside of the tree to the inside.

Equation \eqref{eq: Time evolved cond prob. repr. through BL} shows that, once the configuration $\eta$ at time $t$ is fixed, the transition probabilities $P_t(\omega_x,\eta_x)$ enter as site-dependent local fields on the initial layer. Hence, the tree structure again allows the subtree contributions to be combined recursively, now through boundary messages that depend on the observed spins at time $t$. This induces the following \textit{time-dependent inhomogeneous normalized recursion}. 
For the second-layer spin value $k\in S$, we define the recursion $G_t^{[k]}:(\Delta^{q-1})^d\to\Delta^{q-1}$ by
\begin{equation}\label{eq: Inhomogeneous time-dependent BL equation}
    G_t^{[k]}(v^{(1)},\dots,v^{(d)}):=\frac{P_t(\cdot,k)\odot F(v^{(1)},\dots,v^{(d)})}{\left\|P_t(\cdot,k)\odot F(v^{(1)},\dots,v^{(d)})\right\|_1}
\end{equation}
where $P_t(\cdot,k)\in \Delta^{q-1}$ is defined by $P_t(\cdot,k)(j):=P_t(j,k)$ for all $j\in S$. Let $\eta\in \Omega$ and $\Lambda \subset V$ be connected. Specifically, we introduce \textit{time- and $\eta_\Lambda$-dependent boundary messages} recursively on the subtree
$\overline{\Lambda}=\Lambda\cup\partial\Lambda$ as follows. For directed edges
$\langle x,y\rangle$ with $x\in\partial\Lambda$ and $y\in\Lambda$, we set
$l^t_{xy}[\eta_\Lambda]=l$. For $\langle x,y \rangle$ with
$x,y\in\Lambda$ and $d(x,o)=d(y,o)+1$, we define
\begin{equation}\label{eq: Definition time evolved BL}
    l_{xy}^t[\eta_\Lambda] :=G_t^{[\eta_x]}\big(l^t_{z_1x}[\eta_\Lambda],\dots,l^t_{z_dx}[\eta_\Lambda]\big)
\end{equation}
where $ \{z_1,\dots,z_d\}=\partial x\setminus\{y\}.$
Furthermore, for a given configuration $\eta$, we write 
$G_t^{\eta}:=G_t^{[\eta_x]}$ for the application of \eqref{eq: Definition time evolved BL} at successive sites $x$, compare also Figure \ref{fig: Long-time recovery idea of proof}.

Carrying out the summation in \eqref{eq: Time evolved cond prob. repr. through BL} over the first-layer spin values successively for each site, beginning at the boundary law for the sites in the outer boundary $\partial \Lambda$ and working inwards to the root, leads to
\begin{equation}\label{eq: Root marginals recursive description}
    \mu^l_t(\sigma_o=\eta_o|\sigma_{\Lambda\setminus \{o\}}=\eta_{\Lambda\setminus\{o\}})=\frac{\sum_{\omega_o\in S} P_t(\omega_o, \eta_o) \prod_{x\in \partial \{o\}} Q(\omega_x,\omega_o)l^{t}_{xo}[\eta_\Lambda](\omega_x)}{\sum_{\omega_o\in S}  \prod_{x\in \partial \{o\}} Q(\omega_x,\omega_o)l^{t}_{xo}[\eta_\Lambda](\omega_x)}.
\end{equation}
For fixed $\eta_o$, the dependence of the root conditional probability on the conditioning outside the root is encoded by the incoming messages $l^t_{xo}[\eta_\Lambda]$. Lemma \ref{lem: Upper bound root marginals} in \hyperref[Appendix: Lipschitz bound for root marginals]{Appendix B} gives a Lipschitz upper estimate for this dependence.

\begingroup

\begin{table}[ht]
\scriptsize
    \centering
    \begin{tabular}{c||c|c}
    \rowcolor{gray!50}

        \multicolumn{1}{c||}{\vtop{\hbox{\thead{}}}}& \thead{time-independent}  &  \vtop{\hbox{\thead{time-dependent}}} \\\hline\hline
        \thead{homogeneous\\
        recursions from\\ $\Delta^{q-1}$ to $\Delta^{q-1}$}& \thead{$f$ defined\\ in \eqref{eq: Homogeneous time-independent FP equation}} &\thead{$g_t^{[k]}$ defined\\
        in \eqref{eq: Homogeneous time-dependent recursion}}  \\ \hline
        \rule{0pt}{6mm}
   \multirow{2}{*}[-0.5mm]{\thead{inhomogeneous\\ recursions from\\ $(\Delta^{q-1})^d$ to $\Delta^{q-1}$}} & \multirow{2}{*}[-2mm]{\thead{$F$ defined\\ in \eqref{eq: Inhomogeneous time-independent BL equation}}}&\thead{ $G_t^{[k]}$ spatially hom.\\ configuration $\eta \equiv k$\\
   defined in \eqref{eq: Inhomogeneous time-dependent BL equation}} \\  \cline{3-3}
   && \thead{$G_t^{\eta}$ spatially inhom.\\ configuration $\eta$\\
   defined in \eqref{eq: Definition time evolved BL}}\\ \hline
   \multirow{2}{*}[-3mm]{\thead{relations between\\
   recursions}}&  \thead{$f(v)=F(v,\dots,v)$} & \thead{$g_t^{[k]}(v)=G_t^{[k]}(v,\dots,v)$}\\ \cline{2-3}
   &\multicolumn{2}{c}{\thead{$G_t^{[k]}=H_t^{[k]}\circ F$\\ $H_t^{[k]}$ defined in \eqref{eq: definition H_t^[k]}}}

    \end{tabular}\medskip
    \caption{Overview of the normalized recursions and their relations. The recursions describe how boundary laws or boundary messages are propagated towards the root $o$. The columns distinguish between the time-independent and time-dependent recursions, while the rows distinguish between homogeneous updates, inhomogeneous updates, and relations between recursions. If we choose the same messages as input for the inhomogeneous recursions (second column), we obtain the homogeneous self maps on $\Delta^{q-1}$ (first column). The time-independent recursions (first row) depend only on the transfer matrix $Q$, whereas the time-dependent recursions (second row), additionally contain the influence of the time-dependent spin flip with transition matrix $P_t$ introduced in \eqref{eq: time-dependent spin flip}. The time-dependent inhomogeneous recursion $G_t^{[k]}$ is related to the time-indepdendent recursion $F$ via the map $H_t^{[k]}$ which describes the influence of the time-dependent spin flip.}
    \label{tab: overview recursions}
\end{table}   
\endgroup

\subsection*{Boundary insensitivity via tree contractions}

As a consequence boundary insensitivity in \eqref{eq: Goodness of measures definition} follows once these time-dependent boundary messages lose their dependence on distant boundary condition variations. It therefore suffices to show for quasilocal Gibbsiannness that, for every $\eta\in\Omega$ and every $x\in \partial \{o\}$,
\begin{equation}\label{eq: decay of b.c. influence on boundary law}
    \limsup_{\Lambda \uparrow V}\sup_{\substack{\xi^1,\xi^2\in \Omega\\\Delta: ~\Lambda \subset \Delta \Subset V}} \Big\| l^t_{xo}[\eta_{\Lambda \setminus \{o\}} \xi^1_{\Delta \setminus \Lambda}]-l^t_{xo}[\eta_{\Lambda \setminus \{o\}} \xi^2_{\Delta \setminus \Lambda}]\Big\|_1=0.
\end{equation}
We will do so by showing that, for sufficiently large times, the boundary messages generated by \eqref{eq: Definition time evolved BL}
remain in a neighborhood of the
initial boundary law $l$, where the recursion is uniformly contractive, uniformly in the
observed spin values $\eta_x$. Therefore, we introduce the notation of tree contractions,
which provides a tool to carry out this argument.

\begin{definition}[Tree contraction]\label{def: s-tree contractions}
 Let $(U,d)$ be a metric space. A function
$G:U^s\to U$ is an \underline{$s$-tree contraction} on $U$ if there exists
$\lambda\in[0,1)$ such that
\begin{equation}\label{eq: tree contraction estimate}
d\left(G(x_1,\dots,x_s),G(y_1,\dots,y_s)\right)
\leq \lambda \max_{i=1,\dots,s} d(x_i,y_i)
\end{equation}
for all $x_1,\dots,x_s,y_1,\dots,y_s\in U$. The constant $\lambda$ is called the \underline{contraction constant}.
\end{definition}

Let $l\in\Delta^{q-1}$ be an $f$-stable fixed point of the homogeneous time-independent boundary-law equation $f(l)=l$ where $f$ is defined in \eqref{eq: Homogeneous time-independent FP equation}. Note that, as the simplex $\Delta^{q-1}$ is a flat space, its tangent space 
\begin{equation*}
    T_l\Delta^{q-1}=\Bigl\{x\in \R^q:~\sum^q_{i=1}x_i=0\Bigr\}.
\end{equation*}
is independent of the base point $l\in \mathrm{ri}\left(\Delta^{q-1}\right)$ where 
\begin{equation}\label{eq: Relative interior Simplex}
    \mathrm{ri}(\Delta^{q-1}):=\Bigl\{v=(v_1,\dots,v_q)\in (0,1)^q:~\sum^q_{i=1}v_i=1\Bigr\}
\end{equation}
denotes the 
\textit{relative interior} of the simplex $\Delta^{q-1}$,
and we write $T\Delta^{q-1}:=T_{l}(\Delta^{q-1})$. Furthermore, for every $z^{(1)},\dots,z^{(d)}\in \mathrm{ri}\left(\Delta^{q-1}\right)$, we have $T_{(z^{(1)},\dots,z^{(d)})}\left((\Delta^{q-1})^d\right)=\prod^d_{i=1}T_{z^{(i)}}(\Delta^{q-1})=\left(T\Delta^{q-1}\right)^d$. The key point is that one can choose an $l$-dependent norm on $T\Delta^{q-1}$ with respect to which the time-dependent inhomogeneous recursions are locally tree contractive near $l$.

\begin{remark}\label{rk: Equivalence spectral radius and operator norm}
Since $f:\Delta^{q-1}\rightarrow \Delta^{q-1}$, the differential of $f$ at $l$ is considered as a linear operator $Df(l):T\Delta^{q-1} \rightarrow T\Delta^{q-1}$.
    The condition in Theorem
    \ref{thm:  Goodness at large times} that all eigenvalues of $Df(l)$ have modulus strictly smaller than one is equivalent to the existence of an abstract
    norm $\|\cdot\|_l$ on $T\Delta^{q-1}$ (possibly dependent on $l$) such that the following holds: 
    The induced operator norm satisfies
    \begin{equation*}
    \|Df(l)\|_{l}:=\sup_{\substack{v\in T\Delta^{q-1}\\v\neq(0,\dots,0)}}\frac{\|Df(l)v\|_l}{\|v\|_l}<1,
    \end{equation*}
    see \cite{Se02} Theorem 4.2.1.
    Note that, for simplicity, we denote the operator norm by the same symbol. 
       
\end{remark}

 Both norms will be used in particular in the proof of Theorem \ref{thm: Goodness at large times}, as 
with these norms, common contraction neighborhoods can be constructed on which the time-dependent inhomogeneous recursions are tree-contractions for all sufficiently large times. The main technical input to prove Theorem \ref{thm: Goodness at large times} is the following proposition.

\begin{proposition}[Time-dependent family of maps $G_t$: tree contraction on small balls at large times]
\label{prop: time-dependent map: tree contraction}
    Let $\lambda$ denote the spectral radius of $Df(l)$ and fix $\lambda_G\in (\lambda,1)$. For each $r>0$, set $U_r=U_r(l):=\{x\in\Delta^{q-1}:\|x-l\|_l<r\}$. Then there exists a radius $r_0=r_0(\lambda_G)>0$ such that, for every $r\in (0,r_0)$, there exists a time threshold $t_R=t_R(r,\lambda_G)\geq 0$ with the following properties: for every $t\geq t_R$ and every $k\in S$, the map $G_t^{[k]}$ is a $d$-tree contraction on $U_r$ with contraction constant $\lambda_G$, where $U_r$ is equipped with the metric induced by $\|\cdot\|_l$.
    Moreover, the set $U_r$ is invariant under each map $G_t^{[k]}$ in the sense that $G_t^{[k]}\left((U_r)^d\right)\subset U_r$. 
\end{proposition}

The proof of Proposition \ref{prop: time-dependent map: tree contraction} is presented in Subsection \ref{Subsec: Proof of time dependent tree contraction}. It is in turn based on Lemma \ref{lem: time-independent map: tree contraction}, formulated and proved in Subsection \ref{Subsec: Proof of time independent tree contraction}, which demonstrates the tree-contraction property and invariance of sufficiently small neighborhoods of $l$ under the time-independent inhomogeneous recursion $F$ defined in \eqref{eq: Inhomogeneous time-independent BL equation}. Assuming Proposition \ref{prop: time-dependent map: tree contraction}, we are ready to give the following proof.

\begin{proof}[Proof of Theorem \ref{thm: Goodness at large times}]
    Fix $\lambda_G\in (\lambda,1)$ where $\lambda$ denotes the spectral radius of $Df(l)$. By Proposition \ref{prop: time-dependent map: tree contraction}, there exists $r_0>0$ and for each $r\in (0,r_0)$ a time $t_R\geq 0$ such that for every $t\geq t_R$ and every $k\in S$, the map $G_t^{[k]}$ is a $d$-tree contraction on $U_r$ with contraction constant $\lambda_G$. Let $r\in (0,r_0)$ and $t\geq t_R$ be fixed. 
    
    By the preceding discussion, it suffices to compare two message recursions with second-layer configurations $\eta_{\Lambda\setminus\{o\}}\xi^1_{\Delta\setminus\Lambda}$ and $\eta_{\Lambda\setminus\{o\}}\xi^2_{\Delta\setminus\Lambda}$, see \eqref{eq: Definition time evolved BL} and \eqref{eq: decay of b.c. influence on boundary law}.
Outside $\Delta$, both recursions are initialized with the same boundary law $l$. In the annulus $\Delta\setminus\Lambda$, there are two  recursions given by $G_t^{\xi^1}$ and $G_t^{\xi^2}$ and they may differ through the different configurations $\xi^1$ and $\xi^2$, respectively, at time $t$. By the invariance statement in Proposition \ref{prop: time-dependent map: tree contraction}, all messages produced by both recursions remain in $U_r$.

Inside $\Lambda$, the second-layer configuration is the same for both recursions, namely $\eta_{\Lambda\setminus\{o\}}$. Thus, along each vertex of the recursion inside $\Lambda$, both systems use the same family of maps $G_t^{\eta}$. Therefore, the only difference between the two systems comes from the messages entering $\Lambda$ from its outer boundary. For an illustration of this mechanism, see Figure \ref{fig: Long-time recovery idea of proof}.

\begin{figure}[ht]
\centering
\begin{tikzpicture}[scale=5]

\begin{scope}[shift={(0.7,0)}, scale=0.8]


\def\rOne{0.10}
\def\rTwo{0.22}
\def\rThree{0.34}

\coordinate (root) at (0,0);

\coordinate (v1) at (90:\rOne);
\coordinate (v2) at (210:\rOne);
\coordinate (v3) at (330:\rOne);

\coordinate (w1) at (120:\rTwo);
\coordinate (w2) at (60:\rTwo);

\coordinate (w3) at (180:\rTwo);
\coordinate (w4) at (240:\rTwo);

\coordinate (w5) at (300:\rTwo);
\coordinate (w6) at (0:\rTwo);

\coordinate (x1)  at (135:\rThree);
\coordinate (x2)  at (105:\rThree);

\coordinate (x3)  at (75:\rThree);
\coordinate (x4)  at (45:\rThree);

\coordinate (x5)  at (195:\rThree);
\coordinate (x6)  at (165:\rThree);

\coordinate (x7)  at (225:\rThree);
\coordinate (x8)  at (255:\rThree);

\coordinate (x9)  at (315:\rThree);
\coordinate (x10) at (285:\rThree);

\coordinate (x11) at (345:\rThree);
\coordinate (x12) at (15:\rThree);

\draw[-, thick] (root) -- (v1);
\draw[-, thick] (root) -- (v2);
\draw[-, thick] (root) -- (v3);

\draw[-, thick] (v1) -- (w1);
\draw[-, thick] (v1) -- (w2);

\draw[-, thick] (v2) -- (w3);
\draw[-, thick] (v2) -- (w4);

\draw[-, thick] (v3) -- (w5);
\draw[-, thick] (v3) -- (w6);

\draw[-, thick] (w1) -- (x1);
\draw[-, thick] (w1) -- (x2);

\draw[-, thick] (w2) -- (x3);
\draw[-, thick] (w2) -- (x4);

\draw[-, thick] (w3) -- (x5);
\draw[-, thick] (w3) -- (x6);

\draw[-, thick] (w4) -- (x7);
\draw[-, thick] (w4) -- (x8);

\draw[-, thick] (w5) -- (x9);
\draw[-, thick] (w5) -- (x10);

\draw[-, thick] (w6) -- (x11);
\draw[-, thick] (w6) -- (x12);

\fill (root) circle (0.012);

\fill (v1) circle (0.012);
\fill (v2) circle (0.012);
\fill (v3) circle (0.012);

\fill (w1) circle (0.012);
\fill (w2) circle (0.012);
\fill (w3) circle (0.012);
\fill (w4) circle (0.012);
\fill (w5) circle (0.012);
\fill (w6) circle (0.012);

\fill (x1) circle (0.012);
\fill (x2) circle (0.012);
\fill (x3) circle (0.012);
\fill (x4) circle (0.012);
\fill (x5) circle (0.012);
\fill (x6) circle (0.012);
\fill (x7) circle (0.012);
\fill (x8) circle (0.012);
\fill (x9) circle (0.012);
\fill (x10) circle (0.012);
\fill (x11) circle (0.012);
\fill (x12) circle (0.012);

\draw[color=cyan,  very thick](0,0) circle (0.3);

\draw[color=red, dashed, very thick](0,0) circle (0.8);

\scalebox{1}{
\node[label=:{\textbf{Second layer as selector?}}] () at (0,-1.1) {};}

\scalebox{1}{
\node[label=:{{\color{red}Non-local influence?}}] () at (0,0.5) {};}

\scalebox{1}{
\node[label=:{{\color{cyan}$\eta_\Lambda$ arbitrary}}] () at (0,-0.6) {};}

  \draw[->, very thick, cyan, bend left=30] (0,-0.43) to (0,-0.2);

\scalebox{1.1}{
\node[label=:{$\xi^1_{\Delta \setminus \Lambda}$}] () at (-0.37,0.28) {};}
\scalebox{1.1}{
\node[label=:{$\xi^2_{\Delta \setminus \Lambda}$}] () at (0.2,0.28) {};}

\scalebox{1}{
\node[label=:{{$o$}}] () at (-0.05,-0.05) {};}

\scalebox{1}{
\node[label=:{\textbf{versus}}] () at (0,0.37) {};}

\end{scope}

\begin{scope}[shift={(-0.7,0)}, scale=0.8]

\begin{scope}[shift={(-0.15,0.9)}, scale=0.35]

\coordinate (A) at (0,0);
\coordinate (B) at (1,0);
\coordinate (C) at (0.5,{sqrt(3)/2});

\draw[thick] (A) -- (B) -- (C) -- cycle;

\pgfmathsetmacro{\thetaVal}{5}
\pgfmathsetmacro{\disc}{(\thetaVal-1)^2 - 8}
\pgfmathsetmacro{\xminus}{((\thetaVal-1 - sqrt(\disc))/2)^2}
\pgfmathsetmacro{\xplus}{((\thetaVal-1 + sqrt(\disc))/2)^2}


\coordinate (Lp1) at (0.25,0.15);
\coordinate (Lp2) at (0.75,0.15);
\coordinate (Lp3) at (0.5,0.6);

\coordinate (Lm1) at (barycentric cs:A=\xminus,B=1,C=1);
\coordinate (Lm2) at (barycentric cs:A=1,B=\xminus,C=1);
\coordinate (Lm3) at (barycentric cs:A=1,B=1,C=\xminus);

\coordinate (Lf) at (barycentric cs:A=1,B=1,C=1);

\fill[orange] (Lp1) circle (0.04);
\fill[red] (Lp2) circle (0.02);
\fill[red] (Lp3) circle (0.02);

\fill[black] (Lm1) circle (0.020);
\fill[black] (Lm2) circle (0.020);
\fill[black] (Lm3) circle (0.020);

\fill[black] (Lf) circle (0.020);
\end{scope}

\begin{scope}[shift={(0.2,0.3)}, scale=0.35]

\coordinate (A) at (0,0);
\coordinate (B) at (1,0);
\coordinate (C) at (0.5,{sqrt(3)/2});

\draw[thick] (A) -- (B) -- (C) -- cycle;

\pgfmathsetmacro{\thetaVal}{5}
\pgfmathsetmacro{\disc}{(\thetaVal-1)^2 - 8}
\pgfmathsetmacro{\xminus}{((\thetaVal-1 - sqrt(\disc))/2)^2}
\pgfmathsetmacro{\xplus}{((\thetaVal-1 + sqrt(\disc))/2)^2}


\coordinate (Lp0) at (0.25,0.15);
\coordinate (Lp1) at (0.33,0.15);
\coordinate (Lp2) at (0.75,0.15);
\coordinate (Lp3) at (0.5,0.6);

\coordinate (Lm1) at (barycentric cs:A=\xminus,B=1,C=1);
\coordinate (Lm2) at (barycentric cs:A=1,B=\xminus,C=1);
\coordinate (Lm3) at (barycentric cs:A=1,B=1,C=\xminus);

\coordinate (Lf) at (barycentric cs:A=1,B=1,C=1);

\fill[red] (Lp0) circle (0.02);

\draw[color=orange,  thick](0.25,0.15) circle (0.13);
\fill[orange] (Lp1) circle (0.04);
\fill[red] (Lp2) circle (0.02);
\fill[red] (Lp3) circle (0.02);

\fill[black] (Lm1) circle (0.02);
\fill[black] (Lm2) circle (0.02);
\fill[black] (Lm3) circle (0.02);

\fill[black] (Lf) circle (0.02);
\end{scope}

\begin{scope}[shift={(-0.55,0.3)}, scale=0.35]

\coordinate (A) at (0,0);
\coordinate (B) at (1,0);
\coordinate (C) at (0.5,{sqrt(3)/2});

\draw[thick] (A) -- (B) -- (C) -- cycle;

\pgfmathsetmacro{\thetaVal}{5}
\pgfmathsetmacro{\disc}{(\thetaVal-1)^2 - 8}
\pgfmathsetmacro{\xminus}{((\thetaVal-1 - sqrt(\disc))/2)^2}
\pgfmathsetmacro{\xplus}{((\thetaVal-1 + sqrt(\disc))/2)^2}


\coordinate (Lp0) at (0.25,0.15);
\coordinate (Lp1) at (0.26,0.22);
\coordinate (Lp2) at (0.75,0.15);
\coordinate (Lp3) at (0.5,0.6);

\coordinate (Lm1) at (barycentric cs:A=\xminus,B=1,C=1);
\coordinate (Lm2) at (barycentric cs:A=1,B=\xminus,C=1);
\coordinate (Lm3) at (barycentric cs:A=1,B=1,C=\xminus);

\coordinate (Lf) at (barycentric cs:A=1,B=1,C=1);

\fill[red] (Lp0) circle (0.02);
\fill[orange] (Lp1) circle (0.04);
\fill[red] (Lp2) circle (0.02);
\fill[red] (Lp3) circle (0.02);

\fill[black] (Lm1) circle (0.02);
\fill[black] (Lm2) circle (0.02);
\fill[black] (Lm3) circle (0.02);

\draw[color=orange,  thick](0.25,0.15) circle (0.13);

\fill[black] (Lf) circle (0.02);
\end{scope}

\def\rOne{0.10}
\def\rTwo{0.22}
\def\rThree{0.34}

\coordinate (root) at (0,0);

\coordinate (v1) at (90:\rOne);
\coordinate (v2) at (210:\rOne);
\coordinate (v3) at (330:\rOne);

\coordinate (w1) at (120:\rTwo);
\coordinate (w2) at (60:\rTwo);

\coordinate (w3) at (180:\rTwo);
\coordinate (w4) at (240:\rTwo);

\coordinate (w5) at (300:\rTwo);
\coordinate (w6) at (0:\rTwo);

\coordinate (x1)  at (135:\rThree);
\coordinate (x2)  at (105:\rThree);

\coordinate (x3)  at (75:\rThree);
\coordinate (x4)  at (45:\rThree);

\coordinate (x5)  at (195:\rThree);
\coordinate (x6)  at (165:\rThree);

\coordinate (x7)  at (225:\rThree);
\coordinate (x8)  at (255:\rThree);

\coordinate (x9)  at (315:\rThree);
\coordinate (x10) at (285:\rThree);

\coordinate (x11) at (345:\rThree);
\coordinate (x12) at (15:\rThree);

\draw[-, thick] (root) -- (v1);
\draw[-, thick] (root) -- (v2);
\draw[-, thick] (root) -- (v3);

\draw[-, thick] (v1) -- (w1);
\draw[-, thick] (v1) -- (w2);

\draw[-, thick] (v2) -- (w3);
\draw[-, thick] (v2) -- (w4);

\draw[-, thick] (v3) -- (w5);
\draw[-, thick] (v3) -- (w6);

\draw[-, thick] (w1) -- (x1);
\draw[-, thick] (w1) -- (x2);

\draw[-, thick] (w2) -- (x3);
\draw[-, thick] (w2) -- (x4);

\draw[-, thick] (w3) -- (x5);
\draw[-, thick] (w3) -- (x6);

\draw[-, thick] (w4) -- (x7);
\draw[-, thick] (w4) -- (x8);

\draw[-, thick] (w5) -- (x9);
\draw[-, thick] (w5) -- (x10);

\draw[-, thick] (w6) -- (x11);
\draw[-, thick] (w6) -- (x12);

\fill (root) circle (0.012);

\fill (v1) circle (0.012);
\fill (v2) circle (0.012);
\fill (v3) circle (0.012);

\fill (w1) circle (0.012);
\fill (w2) circle (0.012);
\fill (w3) circle (0.012);
\fill (w4) circle (0.012);
\fill (w5) circle (0.012);
\fill (w6) circle (0.012);

\fill (x1) circle (0.012);
\fill (x2) circle (0.012);
\fill (x3) circle (0.012);
\fill (x4) circle (0.012);
\fill (x5) circle (0.012);
\fill (x6) circle (0.012);
\fill (x7) circle (0.012);
\fill (x8) circle (0.012);
\fill (x9) circle (0.012);
\fill (x10) circle (0.012);
\fill (x11) circle (0.012);
\fill (x12) circle (0.012);

\draw[color=cyan,  very thick](0,0) circle (0.3);

\draw[color=red, dashed, very thick](0,0) circle (0.8);

\scalebox{1}{
\node[label=:{\textbf{First layer}}] () at (0,-1.1) {};}

\scalebox{1}{
\node[label=:{\textbf{versus}}] () at (0,0.4) {};}

\scalebox{1}{
\node[label=:{\textbf{$f$-stable}}] () at (-0.47,0.95) {};}
\scalebox{1}{
\node[label=:{\textbf{fixed point $l$}}] () at (-0.47,0.85) {};}

\scalebox{1}{
\node[label=:{$G_t^{\xi^1}$}] () at (-0.4,-0.5) {};}

\draw[->, very thick, bend left=0] (-0.4,-0.6) to (-0.2,-0.3);
\scalebox{1}{
\node[label=:{$G_t^{\xi^2}$}] () at (0.4,-0.5) {};}

\scalebox{1}{
\node[label=:{\textbf{versus}}] () at (0,-0.6) {};}

\draw[->, very thick, bend left=0] (0.4,-0.6) to (0.2,-0.3);

\scalebox{0.8}{
\node[label=:{{\color{orange}$U_r$}}] () at (-0.975,0.4) {};}

\scalebox{0.8}{
\node[label=:{{\color{orange}$U_r$}}] () at (-0.025,0.4) {};}

\scalebox{1}{
\node[label=:{{\color{cyan}$\Lambda$}}] () at (0.45,-0.1) {};}

\scalebox{1}{
\node[label=:{{$o$}}] () at (-0.05,-0.05) {};}

\scalebox{1}{
\node[label=:{{\color{red}$\Delta\setminus\Lambda$}}] () at (-0.65,-0.1) {};}
\end{scope}

\end{tikzpicture}

\caption{
The maps $G^{\xi^i}_t$ for $i\in \{1,2\}$ act on the first layer (depicted on the left), for arbitrary choices of 
configurations $\xi^i$ which are taken from the second layer (depicted on the right). 
Entering the red (dashed) zone where hypothetically a selection may be produced, from the outside on the first layer from an $f$-stable fixed point $l$, 
one sees different recursions in the hypothetical selection zone, depending on the choices of 
$\xi^1$ and $\xi^2$ depicted on the right. Under the action of these maps the recursions stay trapped 
in the same neighborhood $U_r$, and there is no effective selection. Entering the region $\Lambda$ from there, the two discrete trajectories become exponentially close in the size of $\text{dist}(\{o\},\Lambda^c)$, uniformly in the arbitrary 
choice of $\eta_{\Lambda}$ 
}
\label{fig: Long-time recovery idea of proof}
\end{figure}
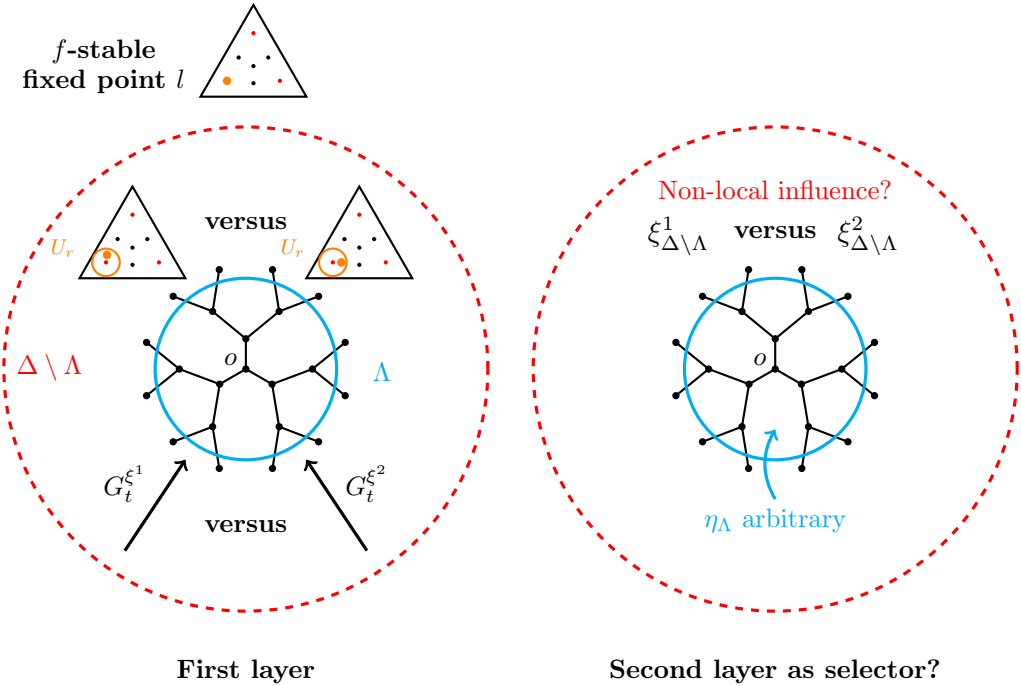

To control the boundary messages in $\Lambda$, we investigate the uniformly contractive family $\{G_t^{[k]}:k\in S\}$ on $U_r$ equipped with the metric induced by $\|\cdot\|_l$. Proposition \ref{prop: time-dependent map: tree contraction} ensures that $G_t^{[k]}\left((U_r)^d\right)\subset U_r$ for every $k\in S$ and that $\lambda_G<1$ as a uniform $d$-tree contraction constant for this family. We set $D:=\text{diam}(U_r)<\infty$.
Further, set $n:=\text{dist}(o,\Lambda^c)$. Lemma \ref{lem: decay influence of inhomogeneous tree contractions} from \hyperref[Appendix: Generalities about tree-indexed contractions]{Appendix A} then implies that the influence at the root of any two boundary data in $U_r$ and located at distance $n$ from the root is bounded by $D\cdot \lambda_G^n$.
Consequently,
\begin{equation}\label{eq: exponential decay of b.c. influence on boundary laws}
    \sup_{\substack{\xi^1,\xi^2\in \Omega\\ \Delta:~\Lambda \subset \Delta \Subset V}}\left\| l^t_{xo}\left[\eta_{\Lambda \setminus \{o\}}\xi^1_{\Delta \setminus \Lambda} \right]-l^t_{xo}\left[\eta_{\Lambda \setminus \{o\}}\xi^2_{\Delta \setminus \Lambda}\right]\right\|_l\leq D\cdot\lambda_G^{\text{dist}(o,\Lambda^c)} 
\end{equation}
for all $x\in \partial \{o\}$. Since $\lambda_G<1$, the right-hand side tends to zero as $\Lambda\uparrow V$. Thus, the claim in \eqref{eq: decay of b.c. influence on boundary law} follows. This proves the single-site version of condition \eqref{eq: Goodness of measures definition} at the root $o$. For the spatially homogeneous measures considered in this paper, the analogous
condition at any other vertex $x\in V$ is obtained by applying a tree automorphism, see Remark \ref{rk: Goodness implies Gibbs} below. This implies \eqref{eq: Goodness of measures definition}, and therefore $\mu^l_t$ is quasilocally Gibbs for all $t\geq t_R$.
\end{proof}

\begin{remark}\label{rk: Goodness implies Gibbs}
One can see by compactness of $\Omega$ that our definition of quasilocal Gibbsianness of the measure $\mu^l_t$, together with the non-nullness already implies that there is a quasilocal specification $\gamma=(\gamma_{\Lambda})_{\Lambda \Subset V}$ in the sense of Gibbsian theory \cite{Ge11}. Here the difference to the above finite-volume probabilities is that 
specification kernels $\omega\mapsto\gamma_{\Lambda}(A|\omega)$ may depend measurably
on all infinitely many spins outside of $\Lambda$. Let us sketch how such a specification can be obtained through a limiting procedure:
Since the spin space is finite, the family of finite-volume conditional kernels $\left(\mu^l_t(\cdot|\eta_\Delta)\right)_{\Delta \Subset V}$ is uniformly bounded and, by the uniformity condition \eqref{eq: Goodness of measures definition},
equicontinuous. By the
Arzelà-Ascoli theorem, see e.g. \cite{Ru87}, subsequential limits therefore exist with $\Delta\uparrow V$. The uniform boundary-insensitivity assumption \eqref{eq: Goodness of measures definition}
gives a unique quasilocal limiting
specification kernel $\gamma_{\{o\}}$. The DLR consistency relations are inherited from the construction as limits 
of finite-volume
conditional probabilities. Non-nullness then yields a non-null quasilocal
specification.

The condition \eqref{eq: Goodness of measures definition} is stated in a single-site form. It first implies that the
single-site conditional probabilities of $\mu^l_t$ admit versions which are
continuous functions of the exterior configuration. More precisely, for $\varphi$ with a support $\{o\}$, the uniform boundary-insensitivity gives a well-defined limit $\gamma_{\{o\}}(\varphi| \eta_{V\setminus\{o\}})$
independent of the auxiliary boundary condition. By translation-invariance of the model, the same holds
at every site $x$. Thus one obtains a family of non-null quasilocal single-site
kernels.

On a finite spin space, such a family determines, provided the usual consistency
relations inherited from the conditional probabilities of $\mu^l_t$, a non-null
quasilocal specification, see \cite{FeMa06}.
\end{remark}

\subsection{Inhomogeneous 
time-independent tree contraction around homogeneous solution}\label{Subsec: Proof of time independent tree contraction}

 The following lemma shows that $f$-stability of $l$ implies that the time-independent inhomogeneous recursion $F$, defined in \eqref{eq: Inhomogeneous time-independent BL equation}, is a $d$-tree contraction on a sufficiently small neighborhood of $l$.

\begin{lemma}[Time-independent map $F$: tree contraction on small balls]
\label{lem: time-independent map: tree contraction}
    Let $\lambda_F\in (\lambda,1)$ where $\lambda$ denotes the spectral radius of $Df(l)$. Then, there exists a radius $r_0=r_0(\lambda_F)>0$ such that, for every $r\in (0,r_0)$, the map $F$ is a $d$-tree contraction on $U_r$ with contraction constant $\lambda_F$, where $U_r$ is equipped with the metric induced by the norm $\|\cdot\|_l$.    
    Here, the set $U_r$ is defined as in Proposition \ref{prop: time-dependent map: tree contraction}, while the $l$-dependent norm $\|\cdot\|_l$ is introduced  abstractly in Remark \ref{rk: Equivalence spectral radius and operator norm}.  
\end{lemma}

From the relation $f(v)=F(v,\dots,v)$ for all $v\in \Delta^{q-1}$ and the permutation invariance of $F$, we first derive the following operator norm estimate for the
differential of $F$ at the diagonal point $(l,\dots,l)$.

\begin{lemma}\label{lem: operator norm estimate for diagonal restriction of F}
 Suppose that
$\|Df(l)\|_{l}=\lambda$. Then, the operator norm of $DF(l,\dots,l)$, induced by $\|\cdot\|_{\mathrm{max}}$ on $T(\Delta^{q-1})^d$ and by $\|\cdot\|_l$ on $T\Delta^{q-1}$, satisfies 
\begin{equation*}
    \|DF(l,\dots,l)\|_{\mathrm{max} \rightarrow l}\leq \lambda.
\end{equation*}
More generally, we write 
\begin{equation*}
\|DF(z^{(1)},\dots,z^{(d)})\|_{\mathrm{max} \rightarrow l}:=\sup_{\substack{(v^{(1)},\dots,v^{(d)})\in T(\Delta^{q-1})^d\\ (v^{(1)},\dots,v^{(d)})\neq (0,\dots,0)}}\frac{\|DF(z^{(1)},\dots,z^{(d)})(v^{(1)},\dots,v^{(d)})\|_l}{\|(v^{(1)},\dots,v^{(d)})\|_{\mathrm{max}}} 
\end{equation*} 
for all $z^{(1)},\dots,z^{(d)}\in\mathrm{ri}\left(\Delta^{q-1}\right)$.
Here, the maximum norm $\|\cdot\|_{\mathrm{max}}$ on $T(\Delta^{q-1})^d$ is defined by
\begin{equation}\label{eq: Product maximum norm}
    \|(v^{(1)},\dots,v^{(d)})\|_{\max}
    :=
    \max_{1\leq i\leq d}\|v^{(i)}\|_l 
\end{equation}
for all $v^{(1)},\dots,v^{(d)}\in T\Delta^{q-1}$.
\end{lemma}

    \begin{proof}
    By definition, the map $F$ is invariant under permutations of its arguments.
    Consequently, since $(l,\dots,l)$ is fixed by every permutation $\alpha\in S_d$, we also have
\begin{equation}\label{eq: Permutation invariance differential}
        DF(l,\dots,l)(v^{(1)},\dots,v^{(d)})=DF(l,\dots,l)(v^{(\alpha(1))},\dots,v^{(\alpha(d))})
    \end{equation}
for every $(v^{(1)},\dots,v^{(d)})\in \left(T\Delta^{q-1}\right)^d$.

 Let $i\in \{1,\dots,d\}$. Define $\iota_i:T\Delta^{q-1}\rightarrow T(\Delta^{q-1})^d$ by $\iota_i(v):=(0,\dots,0,v,0,\dots,0)$
    where $v$ appears at the $i$-th position and $D_iF(l):=DF(l,\dots,l)\circ \iota_i$
    for all $i\in \{1,\dots,d\}$.
Hence, by linearity of $DF(l,\dots,l)$,
\begin{equation}\label{eq: Representation of DF in terms of D_iF}
DF(l,\dots,l)(v^{(1)},\dots,v^{(d)})
=
\sum_{i=1}^d D_iF(l)(v^{(i)}).
\end{equation}
By permutation symmetry, see \eqref{eq: Permutation invariance differential}, we have that $D_iF(l)(v)=D_jF(l)(v)$ for all $i,j\in \{1,\ldots,d\}$. This implies 
    \begin{equation}\label{eq: Df=DF}
        Df(l)v=DF(l,\dots,l)(v,\dots,v)=d  D_1F(l)(v)
    \end{equation}
    for all $v\in T\Delta^{q-1}$. From \eqref{eq: Representation of DF in terms of D_iF} and \eqref{eq: Df=DF}, we obtain 
    \begin{equation*}
        DF(l,\dots,l)(v^{(1)},\dots,v^{(d)})=\frac{1}{d}Df(l)\sum^d_{i=1}v^{(i)}.
    \end{equation*}
  Consequently, the operator norm reads 
    \begin{equation*}
        \|DF(l,\dots,l)\|_{{\mathrm{max} \rightarrow l}}=\frac{1}{d}\sup_{\substack{(v^{(1)},\dots,v^{(d)})\in T(\Delta^{q-1})^d\\(v^{(1)},\dots,v^{(d)})\neq (0,\dots,0)}}\frac{\|Df(l)\sum^d_{i=1}v^{(i)}\|_l}{\|(v^{(1)},\dots,v^{(d)})\|_{\text{max}}}
    \end{equation*}
    where the numerator on the right-hand side satisfies 
    \begin{equation*}
        \Bigl\|Df(l)\sum^d_{i=1}v^{(i)}\Bigr\|_l\leq \|Df(l)\|_{l}\Bigl\|\sum^d_{i=1} v^{(i)}\Bigr\|_l\leq d \|(v^{(1)},\dots,v^{(d)})\|_{\text{max}} \|Df(l)\|_{l}.
    \end{equation*}
    Combining everything leads to $\|DF(l,\dots,l)\|_{\mathrm{max} \rightarrow l}\leq \|Df(l)\|_{l}=\lambda.$
    \end{proof}

   Using Lemma \ref{lem: operator norm estimate for diagonal restriction of F}, we are now able to prove the statement of Lemma \ref{lem: time-independent map: tree contraction}.

   \begin{proof}[Proof of Lemma \ref{lem: time-independent map: tree contraction}]
Choose $\lambda_F\in (\lambda,1)$. Since $Q$ is strictly positive, the map $F$ is continuously differentiable in a neighborhood of $(l,\dots,l)$. Hence, by continuity of $DF$, there exists $r_0>0$ such that 
\begin{equation*}
\sup_{z^{(1)},\dots,z^{(d)}\in U_{r_0}}
\|DF(z^{(1)},\dots,z^{(d)})\|_{\mathrm{max} \rightarrow l}\leq\|DF(l,\dots,l)\|_{\mathrm{max} \rightarrow l}+(\lambda_F-\lambda)<\lambda_F
\end{equation*}
Here, the second inequality follows from Lemma \ref{lem: operator norm estimate for diagonal restriction of F}. Fix $r\in (0,r_0)$, and let $v^{(1)},\dots,v^{(d)}\in U_r$ and $w^{(1)},\dots,w^{(d)}\in U_r$.
It follows that
\begin{align*}
&\|F(v^{(1)},\dots,v^{(d)})-F(w^{(1)},\dots,w^{(d)})\|_l\\
&\leq\int_0^1\left\|DF(z^{(1)}(s),\dots,z^{(d)}(s))(v^{(1)}-w^{(1)},\dots,v^{(d)}-w^{(d)})\right\|_lds \\
&\leq\lambda_F\|(v^{(1)}-w^{(1)},\dots,v^{(d)}-w^{(d)})\|_{\text{max}}.
\end{align*}
Thus, $F$ satisfies the required tree-contraction estimate \eqref{eq: tree contraction estimate} on $U_r$ with contraction constant $\lambda_F<1$. 

Let $r\in (0,r_0)$. It remains to verify that $U_r$ is invariant under $F$ in the sense that $F\left((U_r)^d\right)\subset U_r$.
Choosing $(w^{(1)},\dots,w^{(d)})=(l,\dots,l)$ in the preceding estimate and using $F(l,\dots,l)=l$, we obtain
\begin{equation*}
\|F(v^{(1)},\dots,v^{(d)})-l\|_l\leq\lambda_F\|(v^{(1)}-l,\dots,v^{(d)}-l)\|_{\text{max}}.
\end{equation*}
Since $v^{(i)}\in U_r$, we have
$\|(v^{(1)}-l,\dots,v^{(d)}-l)\|_{\text{max}}<r$. Thus, we have shown that $F\left((U_r)^d\right)\subset U_r$, which completes the proof.
\end{proof}

\subsection{Long-time tree contraction around homogeneous solution}\label{Subsec: Proof of time dependent tree contraction} 

Using the result of Lemma \ref{lem: time-independent map: tree contraction}, the tree-contraction property and invariance of sufficiently small neighborhoods of $l$ can be carried over to the family of time-dependent recursions $\{G_t^{[k]}:~k\in S\}$ at all sufficiently large times, as stated in Proposition \ref{prop: time-dependent map: tree contraction}. The proof of this proposition is split into two parts. First, we verify the contraction estimate \eqref{eq: tree contraction estimate} for each map $G_t^{[k]}$ (Lemma \ref{lem: Long-time contraction estimate}), and then prove the invariance condition $G_t^{[k]}\left((U_r)^d\right)\subset U_r$ (Lemma \ref{lem: Invariance of G_t^[k] on U}).

\begin{lemma}[Tree-contraction on $U_r$]\label{lem: Long-time contraction estimate}
Let $\lambda_G\in (\lambda,1)$. Then there exist constants $r_0>0$ and $t_*=t_*(\lambda_G)<\infty$ such that, for every $r\in (0,r_0)$, every $t\geq t_*$, every $k\in S$, and all 
$(v^{(1)},\dots,v^{(d)}),(w^{(1)},\dots,w^{(d)})\in (U_r)^d$, the following estimate holds
\begin{equation}\label{eq: Time-dependent contraction estimate}
    \left\|G_t^{[k]}(v^{(1)},\dots,v^{(d)}) - G_t^{[k]}(w^{(1)},\dots,w^{(d)})\right\|_l\leq\lambda_G \max_{i\in\{1,\dots,d\}}\|v^{(i)}-w^{(i)}\|_l.
\end{equation}
\end{lemma}

\begin{proof}
Fix $\lambda_G\in (\lambda,1)$ and choose $\lambda_F\in (\lambda,\lambda_G)$ . By Lemma \ref{lem: time-independent map: tree contraction}, there exists $r_0>0$ such that for all $r\in (0,r_0)$ the inhomogeneous time-independent recursion $F$, given in \eqref{eq: Inhomogeneous time-independent BL equation}, is a $d$-tree contraction on $U_r$ with contraction constant $\lambda_F$ and $F\left((U_r)^d\right)\subset U_r$. Let $r\in (0,r_0)$, $t\geq0$ and $k\in S$. Define $H_t^{[k]}:\Delta^{q-1}\rightarrow \Delta^{q-1}$ by 
\begin{equation}\label{eq: definition H_t^[k]}
    H_t^{[k]}(v):=\frac{P_t(\cdot,k)\odot v}{\|P_t(\cdot,k)\odot v\|_1}
\end{equation}
for all $v\in \Delta^{q-1}$. Then $G_t^{[k]} = H_t^{[k]}\circ F$ holds for all $t\geq 0$.
Note that each entry of $\left(P_t(i,k)\right)_{i\in S}$ can be written as $P_t(i,k)=\frac{1}{q}\left(1+\varepsilon_{t,i}^{[k]}\right)$ where $\varepsilon_t^{[k]}:=\left(\varepsilon_{t,i}^{[k]}\right)_{i\in S}$ satisfies $\|\varepsilon_t^{[k]}\|_{\text{max}}=\max_{i\in S}|\varepsilon_{t,i}^{[k]}|=(q-1)e^{-\frac{q}{q-1}t}$ for all $k\in S$. Therefore, for $v\in\Delta^{q-1}$, we can write 
\begin{equation}\label{eq: Equivalent representation of H_t^{[k]}}
    H_t^{[k]}(v)=\frac{(I_q+E_t^{[k]})v}{1+\langle \varepsilon_t^{[k]},v\rangle}.
\end{equation}
where $E_t^{[k]}:=\text{diag}\left(\varepsilon_{t,1}^{[k]},\dots,\varepsilon_{t,q}^{[k]}\right)$ and $I_q=\text{diag}(1,\dots,1)$ denotes the $q\times q$ identity matrix.
Since $F\left((U_r)^d\right)\subset U_r$, it suffices to consider $H_t^{[k]}$ on $U_r$. We claim that
\begin{equation}\label{eq: Estimate operator norm Dg-I}
    \sup_{v\in U_r}\left\| DH_t^{[k]}(v)\right\|_{l}=\sup_{v\in U_r}\sup_{\substack{h\in T\Delta^{q-1}\\h\neq(0,\dots,0)}}\frac{\|DH_t^{[k]}(v)h\|_l}{\|h\|_l}
    \leq \rho_t^{[k]},
\end{equation}
for suitable $\rho_t^{[k]}$ which converges to $1$ as $t\rightarrow \infty$. 
Indeed, for
$v\in U_r$ and $h\in T\Delta^{q-1}$, differentiation yields
\begin{equation}\label{eq: DH*h}
    DH_t^{[k]}(v)h=\frac{(I_q+E_t^{[k]})h}{1+\langle \varepsilon_t^{[k]},v\rangle}-\frac{(I_q+E_t^{[k]})v}{\left( 1+\langle \varepsilon_t^{[k]},v\rangle \right)^2}
    \langle \varepsilon_t^{[k]},h\rangle .
\end{equation}
Since $\|\cdot\|_1$ and $\|\cdot\|_l$ are equivalent norms on the finite-dimensional vector space $T\Delta^{q-1}$, there exist constants $c_1,c_2>0$ such that
\begin{equation}\label{eq: Norm equivalences}
    c_1\|h\|_1 \leq \|h\|_l\leq  c_2\|h\|_1
\end{equation}
for all $h\in T\Delta^{q-1}$.
For $v\in\Delta^{q-1}$, we have $  |\langle \varepsilon_t^{[k]},v\rangle|
    \leq
    \|\epsilon_t^{[k]}\|_{\text{max}}$.
Thus, for all sufficiently large $t$, the following inequality holds
\begin{equation}\label{eq: Estimate 1+<epsilon,v>}
    1+\langle \varepsilon_t^{[k]},v\rangle\geq 1- \|\epsilon_t^{[k]}\|_{\text{max}} > 0.
\end{equation}
Moreover, norm equivalence and the diagonal structure of $E_t^{[k]}$ imply
\begin{equation}\label{eq: Estimate E_t*h}
    \|E_t^{[k]}h\|_1\leq\|\epsilon_t^{[k]}\|_{\text{max}}\|h\|_1\leq \frac{1}{c_1} \|\epsilon_t^{[k]}\|_{\text{max}} \|h\|_l
\end{equation}
for all $h\in T\Delta^{q-1}$. Note that, unlike $\|\cdot\|_l$, the $\ell^1$-norm $\|\cdot\|_1$ is used here on the whole space $\R^q$ on the left-hand side.
By a similar argument, one obtains 
\begin{equation}\label{eq: Estimate <epsilon,h>}
    |\langle \varepsilon_t^{[k]},h\rangle|
    \leq
    \frac{1}{c_1}\|\epsilon_t^{[k]}\|_{\text{max}}\|h\|_l.
\end{equation}
Furthermore, for every $v\in \Delta^{q-1}$, we have
\begin{equation}\label{eq: Estimate (I+E)*v}
    \|(I_q+E_t^{[k]})v\|_1
    \leq
    \|v\|_1+\|E_t^{[k]}v\|_1
    \leq 1+ \|\epsilon_t^{[k]}\|_{\text{max}}.
\end{equation}
Since both $DH_t^{[k]}(v)h$ and $h$ belong to $T\Delta^{q-1}$, their difference belongs to $T\Delta^{q-1}$ as well. Hence, by \eqref{eq: Norm equivalences}, substituting the expression for $DH_t^{[k]}(v)h$ obtained in \eqref{eq: DH*h}, and applying the estimates \eqref{eq: Estimate 1+<epsilon,v>}-\eqref{eq: Estimate (I+E)*v}, results in
\begin{align*}
    \left\|DH_t^{[k]}(v)h-h\right\|_l&\leq c_2\left\|DH_t^{[k]}(v)h-h\right\|_1\\
    &\leq \frac{c_2}{c_1}\left(\frac{2 \|\epsilon_t^{[k]}\|_{\text{max}}}{1- \|\epsilon_t^{[k]}\|_{\text{max}}}+\frac{\left(1+ \|\epsilon_t^{[k]}\|_{\text{max}}\right) \|\epsilon_t^{[k]}\|_{\text{max}}}{\left(1- \|\epsilon_t^{[k]}\|_{\text{max}}\right)^2}\right)\|h\|_l.
\end{align*}
Therefore, we obtain $\left\|DH_t^{[k]}(v)h\right\|_l\leq \|h\|_l +\left\|DH_t^{[k]}(v)h-h\right\|_l\leq \rho_t^{[k]} \|h\|_l $ where $\rho_t^{[k]}$ may be chosen so that $\lim_{t\rightarrow \infty}\rho_t^{[k]}=1$ uniformly in $k\in S$. This follows from $\lim_{t\rightarrow \infty}\|\epsilon_t^{[k]}\|_{\text{max}}=0$ uniformly in $k\in S$. Consequently, \eqref{eq: Estimate operator norm Dg-I} follows.

Finally, let $\kappa>1$ be chosen so close to $1$ that $\lambda_G=\kappa\lambda_F$.
Since $\rho_t^{[k]}\to 1$ uniformly in $k$, there exists $t_*<\infty$
such that $\rho_t^{[k]}\leq \kappa$
for all $t\geq t_*$ and all $k\in S$. Hence, for all
$v,w\in U_r$, the mean value theorem and the convexity of
$U_r$ imply
\begin{equation*}
    \|H_t^{[k]}(v)-H_t^{[k]}(w)\|_l
    \leq
    \kappa\|v-w\|_l.
\end{equation*}

Now let
$(v^{(1)},\dots,v^{(d)}),(w^{(1)},\dots,w^{(d)})\in (U_r)^d$. Since
$G_t^{[k]}=H_t^{[k]}\circ F$, the above inequality and the tree contraction property of $F$ from Lemma \ref{lem: time-independent map: tree contraction} yield
\begin{equation*}
\begin{aligned}
    \left\|G_t^{[k]}(v^{(1)},\dots,v^{(d)})-G_t^{[k]}(w^{(1)},\dots,w^{(d)})\right\|_l&\leq\kappa\left\|F(v^{(1)},\dots,v^{(d)})-F(w^{(1)},\dots,w^{(d)})\right\|_l\\
    &\leq\kappa\lambda_F\max_{i\in\{1,\dots,d\}}\|v^{(i)}-w^{(i)}\|_l.
\end{aligned}
\end{equation*}
Thus $G_t^{[k]}$ satisfies the tree contraction estimate \eqref{eq: tree contraction estimate} on $U_r$ with contraction constant $\lambda_G=\kappa\lambda_F<1$
for all $t\geq t_*$ uniformly in $k\in S$.
\end{proof}

Using the results of Lemma \ref{lem: Long-time contraction estimate}, we now establish the uniform invariance of $U_r$ under the family of maps $\{G_t^{[k]}:~k\in S\}$ for every $r\in (0,r_0)$.

\begin{lemma}[Invariance of $U_r$]\label{lem: Invariance of G_t^[k] on U}
    Let $\lambda_G\in (\lambda,1)$, and let $r_0>0$ and $t_*<\infty$ be the tree-contraction thresholds given by Lemma \ref{lem: Long-time contraction estimate}. Then, for every $r\in (0,r_0)$, there exists a finite time $t_R=t_R(r,\lambda_G)\geq t_*$ such that $G_t^{[k]}\left((U_r)^d\right)\subset U_r$ for all $t\geq t_R$ and all $k\in S$.
\end{lemma}

\begin{proof}
    Let $r\in (0,r_0)$. We claim that
    \begin{equation}\label{eq: Estimate g_t^{[k]}(x)-x}
    \left\|H_t^{[k]}(v)-v\right\|_l\leq c_2\frac{2\|\epsilon_t^{[k]}\|_{\text{max}}}{1-\|\epsilon_t^{[k]}\|_{\text{max}}}
\end{equation}
for all $v\in U_r$ where $\epsilon_t^{[k]}=\left(\epsilon_{t,i}^{[k]}\right)_{i\in S}$ was defined in \eqref{eq: Equivalent representation of H_t^{[k]}} and the constant $c_2>0$ appears in \eqref{eq: Norm equivalences}. In order to see this, we use the representation of $H_t^{[k]}$ given in \eqref{eq: Equivalent representation of H_t^{[k]}}.
Consequently, we obtain
\begin{equation}\label{eq: Representation g_t^{[k]}(v)-v}
     \|H_t^{[k]}(v)-v\|_l=\frac{\|E_t^{[k]}v-\langle \epsilon_t^{[k]},v\rangle v\|_l}{1+\langle \epsilon_t^{[k]},v\rangle}
\end{equation}
The denominator can be upper bounded by $1-\|\epsilon_t^{[k]}\|_\text{max}$, compare \eqref{eq: Estimate 1+<epsilon,v>}. Note that $\|E_t^{[k]}v\|_1\leq \|\epsilon_t^{[k]}\|_{\text{max}}\|v\|_1=\|\epsilon_t^{[k]}\|_{\text{max}}$ and $\|v\|_1=1$ for all $v\in \Delta^{q-1}$. Together with \eqref{eq: Norm equivalences} and  $|\langle \varepsilon_t^{[k]},v\rangle|
    \leq
    \|\epsilon_t^{[k]}\|_{\text{max}}$, the numerator of \eqref{eq: Representation g_t^{[k]}(v)-v} can be upper bounded by $   c_2\|E_t^{[k]}v-\langle \epsilon_t^{[k]},v\rangle v\|_1\leq 2c_2\|\epsilon_t^{[k]}\|_{\text{max}}$ which ends the proof of \eqref{eq: Estimate g_t^{[k]}(x)-x}.

Using $F(l,\dots,l)=l$ together with \eqref{eq: Estimate g_t^{[k]}(x)-x}, we get
\begin{equation}\label{ineq: Invariance of G_t near BL arbitrary norm}
\begin{aligned}
    \left\|G_t^{[k]}(l,\dots,l)-l\right\|_l=
    \left\|H_t^{[k]}(l)-l\right\|_l \leq c_2 \sup_{k\in S}
    \frac{2\|\epsilon_t^{[k]}\|_{\text{max}}}{1-\|\epsilon_t^{[k]}\|_{\text{max}}}.
\end{aligned}
\end{equation}
By Lemma \ref{lem: Long-time contraction estimate}, for every $t\geq t_*$ and every $k\in S$, the map $G_t^{[k]}$ is a $d$-tree contraction on $U_r$ with the uniform contraction constant $\lambda_G$. Since $\lim_{t\to\infty}\|\epsilon_t^{[k]}\|_{\text{max}}=0$ and $S$ finite,
we may choose $t_R\geq t_*$ sufficiently large such that the right-hand side of \eqref{ineq: Invariance of G_t near BL arbitrary norm} is upper bounded by $(1-\lambda_G)r$
for all $t\geq t_R$. Then, for all $t\geq t_R$, all $k\in S$, and all $ (v^{(1)},\dots,v^{(d)})\in U^d,$
the triangle inequality gives
\begin{equation*}
\begin{aligned}
    \left\|G_t^{[k]}(v^{(1)},\dots,v^{(d)})-l\right\|_l&\leq\left\|G_t^{[k]}(v^{(1)},\dots,v^{(d)})-G_t^{[k]}(l,\dots,l)\right\|_l +\left\|G_t^{[k]}(l,\dots,l)-l\right\|_l.
\end{aligned}
\end{equation*}
The first term is bounded by the contraction property of $G_t^{[k]}$, compare Lemma \ref{lem: Long-time contraction estimate},
while the second one is bounded by
\eqref{ineq: Invariance of G_t near BL arbitrary norm}. Hence
\begin{equation*}
\begin{aligned}
    \left\|G_t^{[k]}(v^{(1)},\dots,v^{(d)})-l\right\|_l &\leq
    \lambda_G\max_{m\in\{1,\dots,d\}}\|v^{(m)}-l\|_l+(1-\lambda_G)r.
\end{aligned}
\end{equation*}
Since $(v^{(1)},\dots,v^{(d)})\in (U_r)^d$, we have $\|v^{(i)}-l\|_l<r$ for all $i\in\{1,\dots,d\}$.
Therefore, we conclude that $ G_t^{[k]}(v^{(1)},\dots,v^{(d)})\in U_r.$ This ends the proof of Lemma \ref{lem: Invariance of G_t^[k] on U}.
\end{proof}

\section{Proof of Theorem \ref{thm: Loss without recovery} (total badness): Unperturbed saddle under perturbed boundary law recursion}\label{sec: Proof of total badness}

The argument is organized into three steps. First, we show that the $f$-saddle or free fixed point $l$ converges under the \textit{time-dependent homogeneous
recursion}
\begin{equation}\label{eq: Homogeneous time-dependent recursion}
    g_t^{[i]}:\Delta^{2}\rightarrow \Delta^{2},~~g_t^{[i]}(v):=G_t^{[i]}(v,v)
\end{equation}
where $G_t^{[i]}$ was introduced in \eqref{eq: Inhomogeneous time-dependent BL equation}, to a fixed point $l^{(i)}(t)$ in the strictly 
$p_i$-dominant chamber, defined in \eqref{eq: p_i-dominant chamber}, for $i\in\{2,3\}$. 
This step is model-specific. 

Second, we prove that
once the discrete trajectory under the homogeneous recursion has entered a sufficiently
small neighborhood $U_i$ of this fixed point, it remains in $U_i$ even
under an inhomogeneous recursion, that is, under a
sequence of maps chosen from the family $(G_t^{[j]})_{j\in \{1,2,3\}}$. 
Finally, we show
that the resulting recursions give rise to root messages, compare \eqref{eq: Root marginals recursive description}, belonging to two uniformly separated trapping regions, and hence imply long-range dependence of the time-evolved model. 
The last two steps are essentially 
based on $f$-stability of the fixed points, which allows perturbation around them for large $t$,  and are therefore more general. 
These steps are summarized in the following proposition and the underlying idea is depicted in Figure \ref{fig: Total badness idea of proof}.

\begin{proposition}\label{prop: Total Badness q=3 Potts model}
  Let $i\in \{2,3\}$ and $\theta>4$. Define the sequence $(l^{(i)}_n(t))_{n\in \N_0}$ recursively by $l^{(i)}_0(t):=l$ and $l^{(i)}_{n+1}(t):=g_t^{[i]}(l_n^{(i)}(t))$ for all $n\in \N_0$. Then there exists a time $t_B< \infty$ such that for every $t\geq t_B$, the sequence has the following properties:
  \begin{enumerate}[label=\alph*)]
      \item The limit $l^{(i)}(t):=\lim_{n\rightarrow \infty}l^{(i)}_n(t)$ exists and is a fixed point of $g_t^{[i]}$.
      \item There exist neighborhoods $U_2$ and $U_3$ of the unperturbed $f$-stable fixed points $l_+^{(2)}$ and $l_+^{(3)}$, respectively, as defined in Lemma \ref{lem: Fixed points normalized simplex Potts model}, such that the limiting fixed points $l^{(2)}(t)$ and $l^{(3)}(t)$ from part a) satisfy $l^{(i)}(t)\in U_i$ for $i\in \{2,3\}$.
Moreover, for every $j\in\{1,2,3\}$ and every $i\in\{2,3\}$, the corresponding time-dependent inhomogeneous recursion satisfies $G_t^{[j]}\left((U_i)^{2}\right)\subset U_i.$
      \item The neighborhoods $U_2$ and $U_3$ can be chosen such that the corresponding root conditional probabilities \eqref{eq: Root marginals recursive description} are uniformly separated. More precisely, there exists $C>0$ such that
      \begin{equation}\label{eq: Uniform separation root marginals}
          \inf_{\substack{l^{(1)},l^{(2)},l^{(3)}\in U_2\\
          \widetilde{l}^{(1)},\widetilde{l}^{(2)},\widetilde{l}^{(3)}\in U_3}}\left|R_{t,2}(l^{(1)},l^{(2)},l^{(3)})-R_{t,2}(\widetilde{l}^{(1)},\widetilde{l}^{(2)},\widetilde{l}^{(3)})\right|\geq C\mathrm{exp}\left(-\frac{3t}{2}\right) 
      \end{equation}
      where $R_{t,2}$ is defined in \eqref{eq: Definition R_t,k and H_a(l)} in the \hyperref[Appendix: Lipschitz bound for root marginals]{Appendix B}.
  \end{enumerate} 
\end{proposition}

\begin{figure}[ht]
    \centering
   \subfloat{\includegraphics[width=0.47\textwidth]{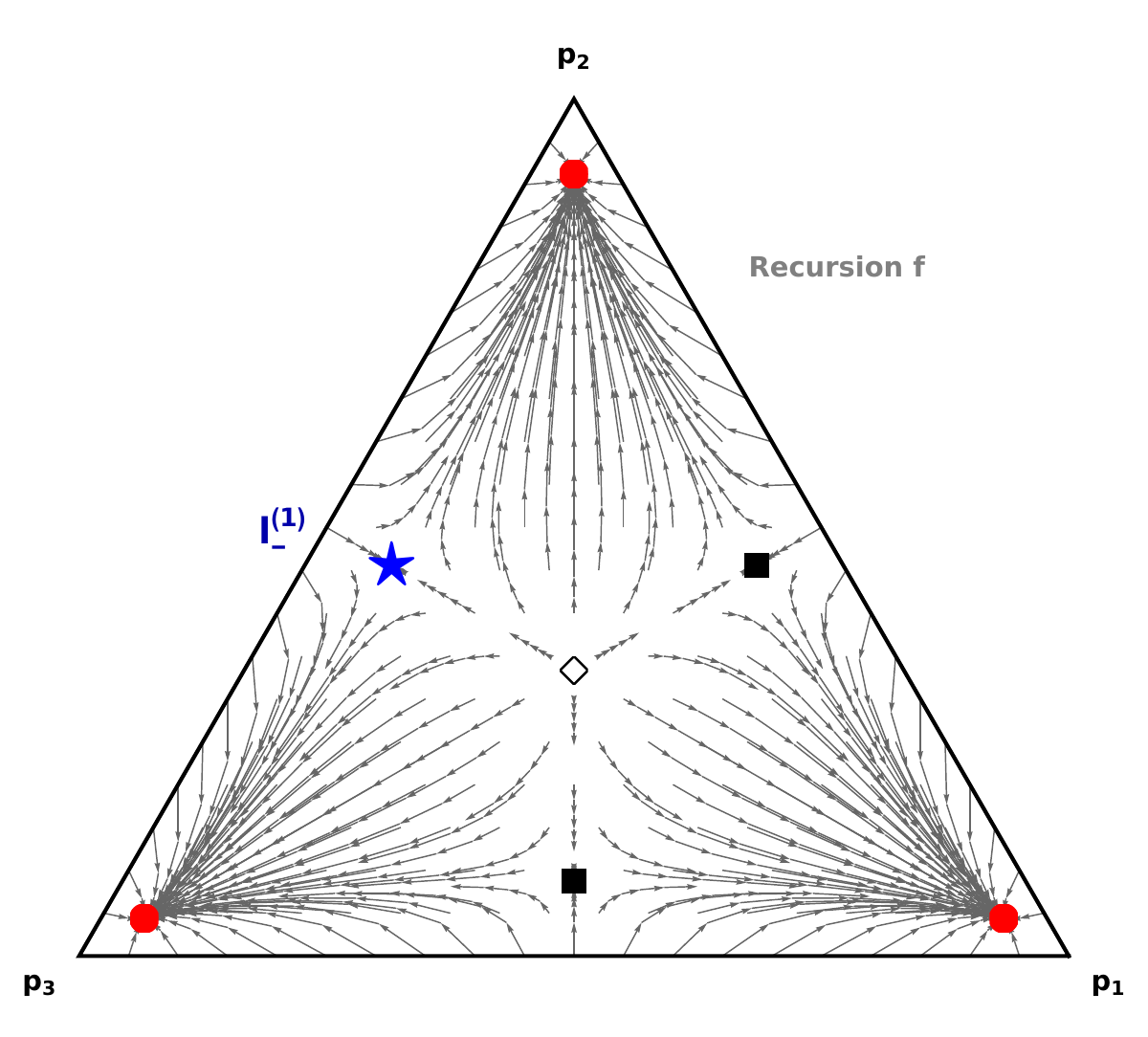}}
  \hfill
   \subfloat{\includegraphics[width=0.47\textwidth]{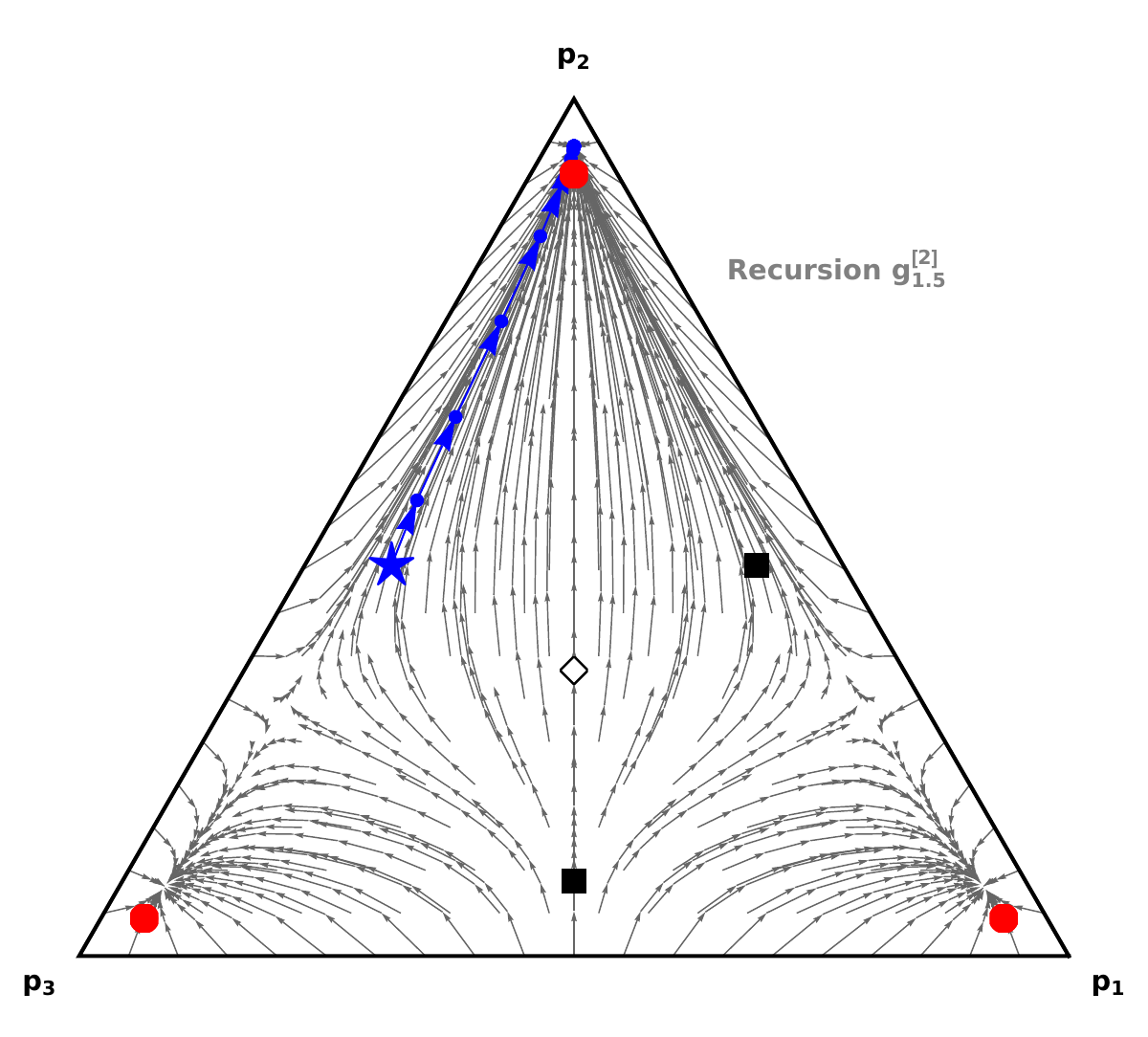}}
    \caption{Depicted are $228$ discrete trajectories (shown as gray polygonal paths) for two different message recursions on the simplex $\Delta^2$ of the three-state Potts model at $\theta=6$. Each consists of $5$ iterations, with starting points symmetrically distributed on a grid in the $2$-simplex. The figure on the left-hand side shows the trajectories of the time-independent homogeneous boundary-law recursion $f$, displayed in \eqref{eq: Normalized homogeneous BL recursion Potts}. In both pictures, the $f$-stable fixed points (marked as red circles), the $f$-saddles with one unstable direction (marked as black squares and $l_-^{(1)}$ as a blue star) and the free fixed point (marked as an unfilled diamond) of $f$ are drawn. The figure on the right-hand side shows the trajectories of the time-dependent homogeneous recursion $g_{1.5}^{[2]}$,  defined in \eqref{eq: Homogeneous time-dependent recursion}. The trajectory of the recursion starting from the saddle point $l_-^{(1)}$ and converging to $p_2$ is highlighted with blue thick arrows and consists of $30$ iterations.}
    \label{fig: Discrete trajectories normalized recursion}
\end{figure}

Part a) will be proved in Subsection \ref{Subsec: Orbit of time-evolved recursion starting from saddle point}, where the time-dependent homogeneous recursion is analyzed and the convergence to the limiting fixed point $l^{(i)}(t)$ is established. Part b) will be proved in Subsection \ref{sec: observation-invariant-set}, where trapping neighborhoods around these limiting fixed points are constructed for the inhomogeneous recursion. Finally, part c) will be proved in Subsection \ref{sec: observation-root-separation}, where the resulting root conditional probabilities associated with the two trapping regions are shown to be uniformly separated.

\begin{proof}[Proof of Theorem \ref{thm: Loss without recovery}]
   We assume that $o\in \Lambda_n$ for all $n\in \N$. Note that Theorem \ref{thm: Loss without recovery} follows by choosing, for each $n\in \N$, a ball $\Delta_n$ around the root $o$ such that $\text{dist}(\Lambda_n,(\Delta_n)^c)\geq N$. Here, we denote by $N=N(t)\in \N$ the minimal number of iterations of the homogeneous recursion $g_t^{[i]}$ required to ensure that $l^{(i)}_n(t)\in U_i$ for all $n\geq N$ and $i\in \{2,3\}$. The existence of such an $N$ follows from part a) of Proposition \ref{prop: Total Badness q=3 Potts model}. Once the homogeneous recursion through the annulus $\Delta_n \setminus \Lambda_n$ has entered the trapping neighborhood $U_i$, it remains in $U_i$ under all subsequent applications of the inhomogeneous recursions $G_t^\eta$ for each $i\in \{2,3\}$, by part b) of Proposition \ref{prop: Total Badness q=3 Potts model}. Finally, part c) of the same proposition implies that the messages reaching the root $o$ yield the root marginals defined in \eqref{eq: Root marginals recursive description}, which are separated in the sense of \eqref{eq: Root marginal separation thm}. The recursion from the exterior of $\Delta_n$ to the root $o$ is described in detail in Section \ref{sec: Proof of long time recovery}.
\end{proof}

\begin{remark}\label{rk: Small Theta region}
    What are possibilities for an unperturbed saddle under $g_t^{[1]}, g_t^{[2]}, g_t^{[3]}$? 

When $t$ is large, the structure of fixed points stays intact, and their basins of attraction deform nicely generically, but 
this is subject to transversality conditions, and is studied in discrete dynamical systems \cite{GuHo90}. 
These conditions would need to be checked in concrete examples. 

What does the three-state Potts model show for smaller $\theta\in (1+2\sqrt{2},4)$?
The pair of $f$-stable fixed points between which the selection game from a saddle can be played, 
may change for different parameter values, but some selection seems always possible, as our numerical investigation indicates. 
In this sense the mechanism of total badness depends on parameter values, but it 
is expected (up to exceptional values) that there should be total badness from saddles. 
Compare the behavior of trajectories for the particular 
choice of parameters presented in Figure \ref{fig: Discrete trajectories normalized recursion small theta}.  
\end{remark}

In the remaining three subsections we give detailed proofs
for the Lemmata used above in the proof of Proposition \ref{prop: Total Badness q=3 Potts model}.

\begin{figure}[ht]
\centering
\begin{tikzpicture}[scale=5]

\begin{scope}[shift={(0.7,0)}, scale=0.8]


\def\rOne{0.10}
\def\rTwo{0.22}
\def\rThree{0.34}

\coordinate (root) at (0,0);

\coordinate (v1) at (90:\rOne);
\coordinate (v2) at (210:\rOne);
\coordinate (v3) at (330:\rOne);

\coordinate (w1) at (120:\rTwo);
\coordinate (w2) at (60:\rTwo);

\coordinate (w3) at (180:\rTwo);
\coordinate (w4) at (240:\rTwo);

\coordinate (w5) at (300:\rTwo);
\coordinate (w6) at (0:\rTwo);

\coordinate (x1)  at (135:\rThree);
\coordinate (x2)  at (105:\rThree);

\coordinate (x3)  at (75:\rThree);
\coordinate (x4)  at (45:\rThree);

\coordinate (x5)  at (195:\rThree);
\coordinate (x6)  at (165:\rThree);

\coordinate (x7)  at (225:\rThree);
\coordinate (x8)  at (255:\rThree);

\coordinate (x9)  at (315:\rThree);
\coordinate (x10) at (285:\rThree);

\coordinate (x11) at (345:\rThree);
\coordinate (x12) at (15:\rThree);

\draw[-, thick] (root) -- (v1);
\draw[-, thick] (root) -- (v2);
\draw[-, thick] (root) -- (v3);

\draw[-, thick] (v1) -- (w1);
\draw[-, thick] (v1) -- (w2);

\draw[-, thick] (v2) -- (w3);
\draw[-, thick] (v2) -- (w4);

\draw[-, thick] (v3) -- (w5);
\draw[-, thick] (v3) -- (w6);

\draw[-, thick] (w1) -- (x1);
\draw[-, thick] (w1) -- (x2);

\draw[-, thick] (w2) -- (x3);
\draw[-, thick] (w2) -- (x4);

\draw[-, thick] (w3) -- (x5);
\draw[-, thick] (w3) -- (x6);

\draw[-, thick] (w4) -- (x7);
\draw[-, thick] (w4) -- (x8);

\draw[-, thick] (w5) -- (x9);
\draw[-, thick] (w5) -- (x10);

\draw[-, thick] (w6) -- (x11);
\draw[-, thick] (w6) -- (x12);

\fill (root) circle (0.012);

\fill (v1) circle (0.012);
\fill (v2) circle (0.012);
\fill (v3) circle (0.012);

\fill (w1) circle (0.012);
\fill (w2) circle (0.012);
\fill (w3) circle (0.012);
\fill (w4) circle (0.012);
\fill (w5) circle (0.012);
\fill (w6) circle (0.012);

\fill (x1) circle (0.012);
\fill (x2) circle (0.012);
\fill (x3) circle (0.012);
\fill (x4) circle (0.012);
\fill (x5) circle (0.012);
\fill (x6) circle (0.012);
\fill (x7) circle (0.012);
\fill (x8) circle (0.012);
\fill (x9) circle (0.012);
\fill (x10) circle (0.012);
\fill (x11) circle (0.012);
\fill (x12) circle (0.012);

\draw[color=cyan,  very thick](0,0) circle (0.3);

\draw[color=red, dashed, very thick](0,0) circle (0.8);

\scalebox{1}{
\node[label=:{\textbf{Second layer as selector}}] () at (0,-1.1) {};}

\scalebox{1}{
\node[label=:{{\color{red}selective annulus}}] () at (0,0.55) {};}

\scalebox{1.1}{
\node[label=:{$\eta_{\Delta\setminus \Lambda} \equiv 2$}] () at (-0.45,0.25) {};}
\scalebox{1.1}{
\node[label=:{$\eta_{\Delta\setminus \Lambda} \equiv 3$}] () at (0.26,0.25) {};}

\scalebox{1}{
\node[label=:{\textbf{versus}}] () at (0,0.325) {};}
\scalebox{1}{
\node[label=:{{\color{cyan}$\eta_\Lambda$ arbitrary}}] () at (0,-0.6) {};}

\scalebox{1}{
\node[label=:{{$o$}}] () at (-0.05,-0.05) {};}

  \draw[->, very thick, cyan, bend left=30] (0,-0.43) to (0,-0.2);
\end{scope}

\begin{scope}[shift={(-0.7,0)}, scale=0.8]
\begin{scope}[shift={(-0.15,0.9)}, scale=0.35]

\coordinate (A) at (0,0);
\coordinate (B) at (1,0);
\coordinate (C) at (0.5,{sqrt(3)/2});

\draw[thick] (A) -- (B) -- (C) -- cycle;

\pgfmathsetmacro{\thetaVal}{5}
\pgfmathsetmacro{\disc}{(\thetaVal-1)^2 - 8}
\pgfmathsetmacro{\xminus}{((\thetaVal-1 - sqrt(\disc))/2)^2}
\pgfmathsetmacro{\xplus}{((\thetaVal-1 + sqrt(\disc))/2)^2}

\coordinate (Lp1) at (barycentric cs:A=\xplus,B=1,C=1);
\coordinate (Lp2) at (barycentric cs:A=1,B=\xplus,C=1);
\coordinate (Lp3) at (barycentric cs:A=1,B=1,C=\xplus);

\coordinate (Lm1) at (barycentric cs:A=\xminus,B=1,C=1);
\coordinate (Lm2) at (barycentric cs:A=1,B=\xminus,C=1);
\coordinate (Lm3) at (barycentric cs:A=1,B=1,C=\xminus);

\coordinate (Lf) at (barycentric cs:A=1,B=1,C=1);

\fill[red] (Lp1) circle (0.02);
\fill[red] (Lp2) circle (0.02);
\fill[red] (Lp3) circle (0.02);

\fill[black] (Lm1) circle (0.020);
\fill[black] (Lm2) circle (0.020);
\fill[black] (Lm3) circle (0.050);

\fill[black] (Lf) circle (0.020);
\end{scope}

\begin{scope}[shift={(0.2,0.3)}, scale=0.35]

\coordinate (A) at (0,0);
\coordinate (B) at (1,0);
\coordinate (C) at (0.5,{sqrt(3)/2});

\draw[thick] (A) -- (B) -- (C) -- cycle;

\pgfmathsetmacro{\thetaVal}{5}
\pgfmathsetmacro{\disc}{(\thetaVal-1)^2 - 8}
\pgfmathsetmacro{\xminus}{((\thetaVal-1 - sqrt(\disc))/2)^2}
\pgfmathsetmacro{\xplus}{((\thetaVal-1 + sqrt(\disc))/2)^2}

\coordinate (Lp1) at (barycentric cs:A=\xplus,B=1,C=1);
\coordinate (Lp2) at (barycentric cs:A=1,B=\xplus,C=1);
\coordinate (Lp3) at (barycentric cs:A=1,B=1,C=\xplus);

\coordinate (Lm1) at (barycentric cs:A=\xminus,B=1,C=1);
\coordinate (Lm2) at (barycentric cs:A=1,B=\xminus,C=1);
\coordinate (Lm3) at (barycentric cs:A=1,B=1,C=\xminus);

\coordinate (Lf) at (barycentric cs:A=1,B=1,C=1);

\fill[red] (Lp1) circle (0.02);
\fill[red] (Lp2) circle (0.02);
\fill[red] (Lp3) circle (0.02);

\fill[black] (Lm1) circle (0.02);
\fill[black] (Lm2) circle (0.02);
\fill[black] (0.81,0.13,0.05) circle (0.050);

\fill[black] (Lf) circle (0.02);
\end{scope}

\begin{scope}[shift={(-0.55,0.3)}, scale=0.35]

\coordinate (A) at (0,0);
\coordinate (B) at (1,0);
\coordinate (C) at (0.5,{sqrt(3)/2});

\draw[thick] (A) -- (B) -- (C) -- cycle;

\pgfmathsetmacro{\thetaVal}{5}
\pgfmathsetmacro{\disc}{(\thetaVal-1)^2 - 8}
\pgfmathsetmacro{\xminus}{((\thetaVal-1 - sqrt(\disc))/2)^2}
\pgfmathsetmacro{\xplus}{((\thetaVal-1 + sqrt(\disc))/2)^2}

\coordinate (Lp1) at (barycentric cs:A=\xplus,B=1,C=1);
\coordinate (Lp2) at (barycentric cs:A=1,B=\xplus,C=1);
\coordinate (Lp3) at (barycentric cs:A=1,B=1,C=\xplus);

\coordinate (Lm1) at (barycentric cs:A=\xminus,B=1,C=1);
\coordinate (Lm2) at (barycentric cs:A=1,B=\xminus,C=1);
\coordinate (Lm3) at (barycentric cs:A=1,B=1,C=\xminus);

\coordinate (Lf) at (barycentric cs:A=1,B=1,C=1);

\fill[red] (Lp1) circle (0.02);
\fill[red] (Lp2) circle (0.02);
\fill[red] (Lp3) circle (0.02);

\fill[black] (Lm1) circle (0.02);
\fill[black] (Lm2) circle (0.02);
\coordinate (nearLeft) at (barycentric cs:A=0.5,B=0.1,C=0.1);
\fill[black] (nearLeft) circle (0.050);

\fill[black] (Lf) circle (0.02);
\end{scope}

\def\rOne{0.10}
\def\rTwo{0.22}
\def\rThree{0.34}

\coordinate (root) at (0,0);

\coordinate (v1) at (90:\rOne);
\coordinate (v2) at (210:\rOne);
\coordinate (v3) at (330:\rOne);

\coordinate (w1) at (120:\rTwo);
\coordinate (w2) at (60:\rTwo);

\coordinate (w3) at (180:\rTwo);
\coordinate (w4) at (240:\rTwo);

\coordinate (w5) at (300:\rTwo);
\coordinate (w6) at (0:\rTwo);

\coordinate (x1)  at (135:\rThree);
\coordinate (x2)  at (105:\rThree);

\coordinate (x3)  at (75:\rThree);
\coordinate (x4)  at (45:\rThree);

\coordinate (x5)  at (195:\rThree);
\coordinate (x6)  at (165:\rThree);

\coordinate (x7)  at (225:\rThree);
\coordinate (x8)  at (255:\rThree);

\coordinate (x9)  at (315:\rThree);
\coordinate (x10) at (285:\rThree);

\coordinate (x11) at (345:\rThree);
\coordinate (x12) at (15:\rThree);

\draw[-, thick] (root) -- (v1);
\draw[-, thick] (root) -- (v2);
\draw[-, thick] (root) -- (v3);

\draw[-, thick] (v1) -- (w1);
\draw[-, thick] (v1) -- (w2);

\draw[-, thick] (v2) -- (w3);
\draw[-, thick] (v2) -- (w4);

\draw[-, thick] (v3) -- (w5);
\draw[-, thick] (v3) -- (w6);

\draw[-, thick] (w1) -- (x1);
\draw[-, thick] (w1) -- (x2);

\draw[-, thick] (w2) -- (x3);
\draw[-, thick] (w2) -- (x4);

\draw[-, thick] (w3) -- (x5);
\draw[-, thick] (w3) -- (x6);

\draw[-, thick] (w4) -- (x7);
\draw[-, thick] (w4) -- (x8);

\draw[-, thick] (w5) -- (x9);
\draw[-, thick] (w5) -- (x10);

\draw[-, thick] (w6) -- (x11);
\draw[-, thick] (w6) -- (x12);

\fill (root) circle (0.012);

\fill (v1) circle (0.012);
\fill (v2) circle (0.012);
\fill (v3) circle (0.012);

\fill (w1) circle (0.012);
\fill (w2) circle (0.012);
\fill (w3) circle (0.012);
\fill (w4) circle (0.012);
\fill (w5) circle (0.012);
\fill (w6) circle (0.012);

\fill (x1) circle (0.012);
\fill (x2) circle (0.012);
\fill (x3) circle (0.012);
\fill (x4) circle (0.012);
\fill (x5) circle (0.012);
\fill (x6) circle (0.012);
\fill (x7) circle (0.012);
\fill (x8) circle (0.012);
\fill (x9) circle (0.012);
\fill (x10) circle (0.012);
\fill (x11) circle (0.012);
\fill (x12) circle (0.012);

\draw[color=cyan,  very thick](0,0) circle (0.3);

\draw[color=red, dashed, very thick](0,0) circle (0.8);

\scalebox{1}{
\node[label=:{\textbf{First layer}}] () at (0,-1.1) {};}

\scalebox{1}{
\node[label=:{\textbf{versus}}] () at (0,0.4) {};}

\scalebox{1}{
\node[label=:{\textbf{$f$-saddle}}] () at (-0.45,0.95) {};}

\scalebox{1}{
\node[label=:{$g_t^{[2]}$}] () at (-0.4,-0.5) {};}

\draw[->, very thick, bend left=0] (-0.4,-0.6) to (-0.2,-0.3);
\scalebox{1}{
\node[label=:{$g_t^{[3]}$}] () at (0.4,-0.5) {};}

\draw[->, very thick, bend left=0] (0.4,-0.6) to (0.2,-0.3);

\scalebox{1}{
\node[label=:{\textbf{versus}}] () at (0,-0.6) {};}

\scalebox{1}{
\node[label=:{{\color{cyan}$\Lambda$}}] () at (0.45,-0.1) {};}

\scalebox{1}{
\node[label=:{{$o$}}] () at (-0.05,-0.05) {};}

\scalebox{1}{
\node[label=:{{\color{red}$\Delta\setminus\Lambda$}}] () at (-0.65,-0.1) {};}
\end{scope}

\end{tikzpicture}

\caption{The map $g^{[i]}_t$ for $i\in \{2,3\}$ acts on the first layer (depicted on the left), for certain choices of 
configurations $\eta$ which are taken from the second layer (depicted on the right). 
Entering the red (dashed) selection zone from the outside on the first layer   
allows to drive the recursion from the $f$-saddle $l$ into neighborhoods either in 
the lower left or right corner of the simplex. Which one depends 
on the second-layer choice of the conditioning $\eta\equiv i$ for $i\in \{2,3\}$ inside the annulus $\Delta \backslash \Lambda$. Discrete trajectories in the blue zone on the first layer stay then trapped in the respective connected components, for arbitrary prescription of $\eta$ inside $\Lambda$. }
\label{fig: Total badness idea of proof}
\end{figure}
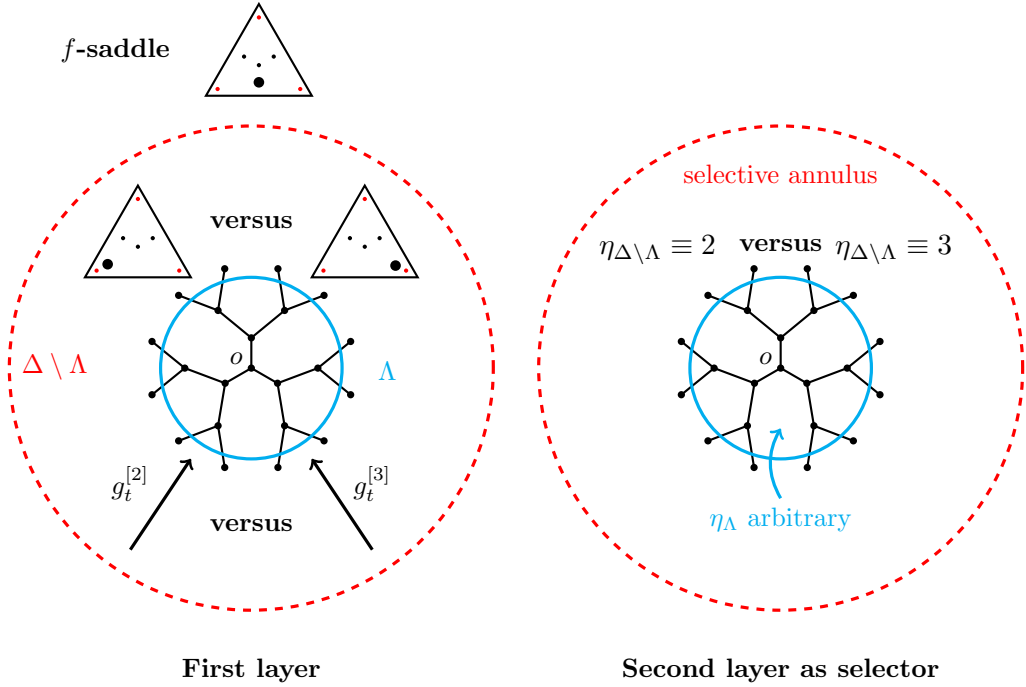

\subsection{Deformations of basins of attraction
}\label{Subsec: Orbit of time-evolved recursion starting from saddle point}

We prove part $a)$ of Proposition \ref{prop: Total Badness q=3 Potts model} for $i=2$. The proof for $i=3$ follows by symmetry. It will be convenient to work in \textit{quotient coordinates} by normalizing the second coordinate. In more detail, we define 
\begin{equation}\label{eq: Quotient coordinates definition}
    (x,y):=\left(\frac{v_1}{v_2},\frac{v_3}{v_2}\right)
\end{equation}
for each $v=(v_1,v_2,v_3)\in (0,\infty)^3$. First, we compute the time-dependent homogeneous recursion $g_t^{[2]}$ in quotient coordinates. We then prove convergence of the discrete trajectory starting from $l$ by combining the monotonicity of the recursion with the invariance of the diagonal. This reduces the problem to a monotone one-dimensional sequence whose trajectory is trapped and therefore converges to a fixed point, see Lemma \ref{lem: Orbits in the p2-dominant chamber}.

\begin{lemma}\label{lem: Time evolved recursion maps quotient coordinates}
For the constant configuration $\eta\equiv k$ where $k\in \{1,2,3\}$, the time-dependent homogeneous
recursion $g_t^{[k]}$ expressed in the quotient coordinates defined in \eqref{eq: Quotient coordinates definition} is represented by the map
$\widetilde{g}_t^{[k]}:(\R_+)^2\to(\R_+)^2$ defined by 
\begin{equation*}
\widetilde{g}_t^{[1]}(x,y)=\left( m_tg_1, g_2\right), \quad \widetilde{g}_t^{[2]}(x,y)=\left( m_t g_1,m_t g_2\right), \quad \widetilde{g}_t^{[3]}(x,y)
=\left( g_1,m_t g_2\right)
\end{equation*}
where
\begin{equation*}
    g_1=g_1(x,y):=\left(\frac{\theta x+y+1}{x+y+\theta}\right)^2\qquad \text{and} \qquad g_2=g_2(x,y):=\left(\frac{x+\theta y+1}{x+y+\theta}\right)^2
\end{equation*}
and $m_t:=\frac{1-e^{-\frac{3}{2} t}}{1+2e^{-\frac{3}{2} t}}.$
\end{lemma}

\begin{proof} 
This is a straightforward computation which is omitted.
\end{proof} 

Discrete trajectories of $g_t^{[2]}$ are displayed in the right panel of Figure \ref{fig: Discrete trajectories normalized recursion}.
Note that $(x_-,1)$ is the representation of the $f$-saddle $l_-^{(1)}$ in quotient coordinates and $(1,1)$ is the representation of the free fixed point $l_{\text{free}}$, compare Lemma \ref{lem: Fixed points normalized simplex Potts model} and \eqref{eq: Quotient coordinates definition}. The goal of this subsection is to prove the following statement.

\begin{figure}[ht]
    \centering
   \subfloat{\includegraphics[width=0.47\textwidth]{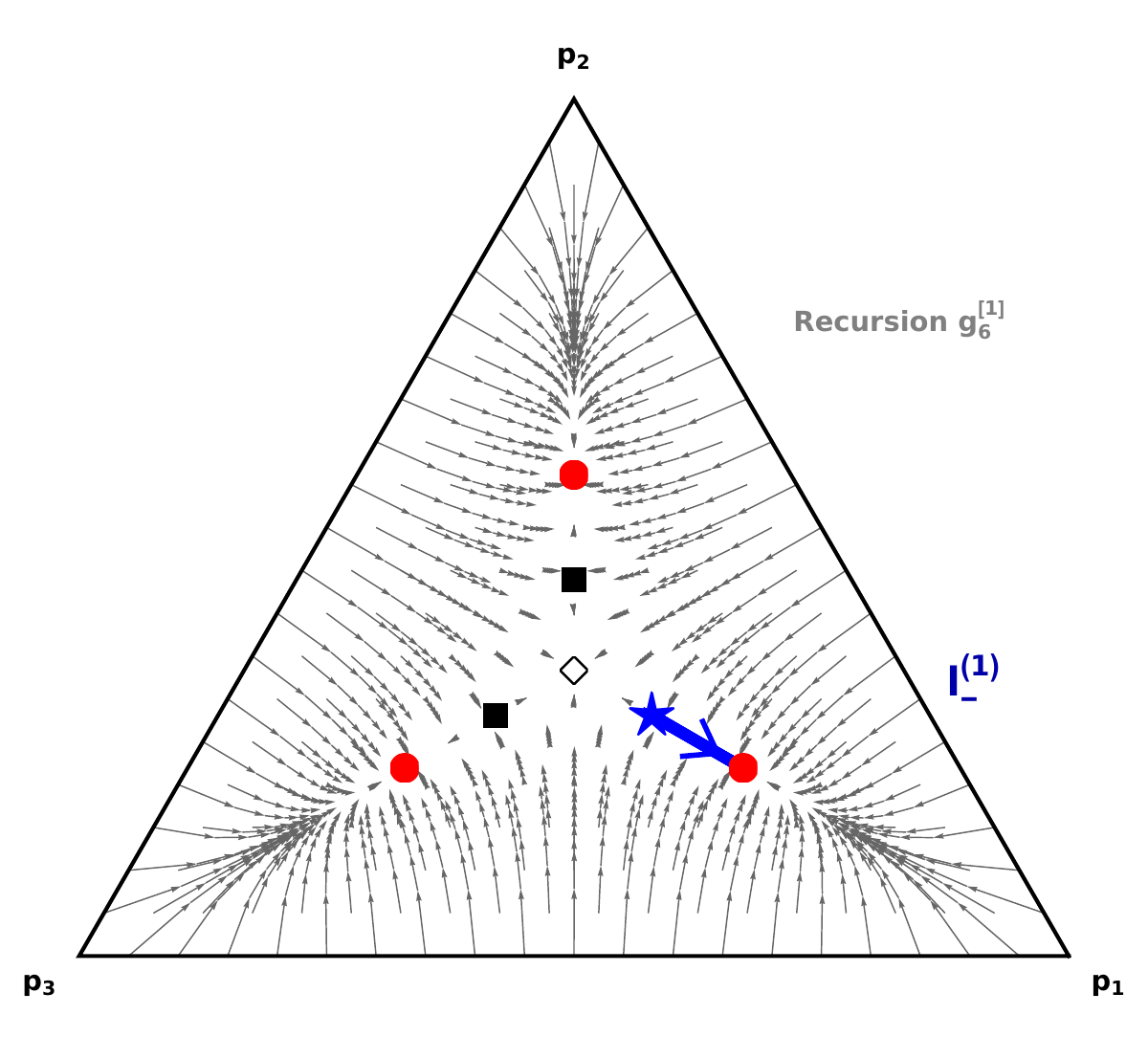}}
  \hfill
   \subfloat{\includegraphics[width=0.47\textwidth]{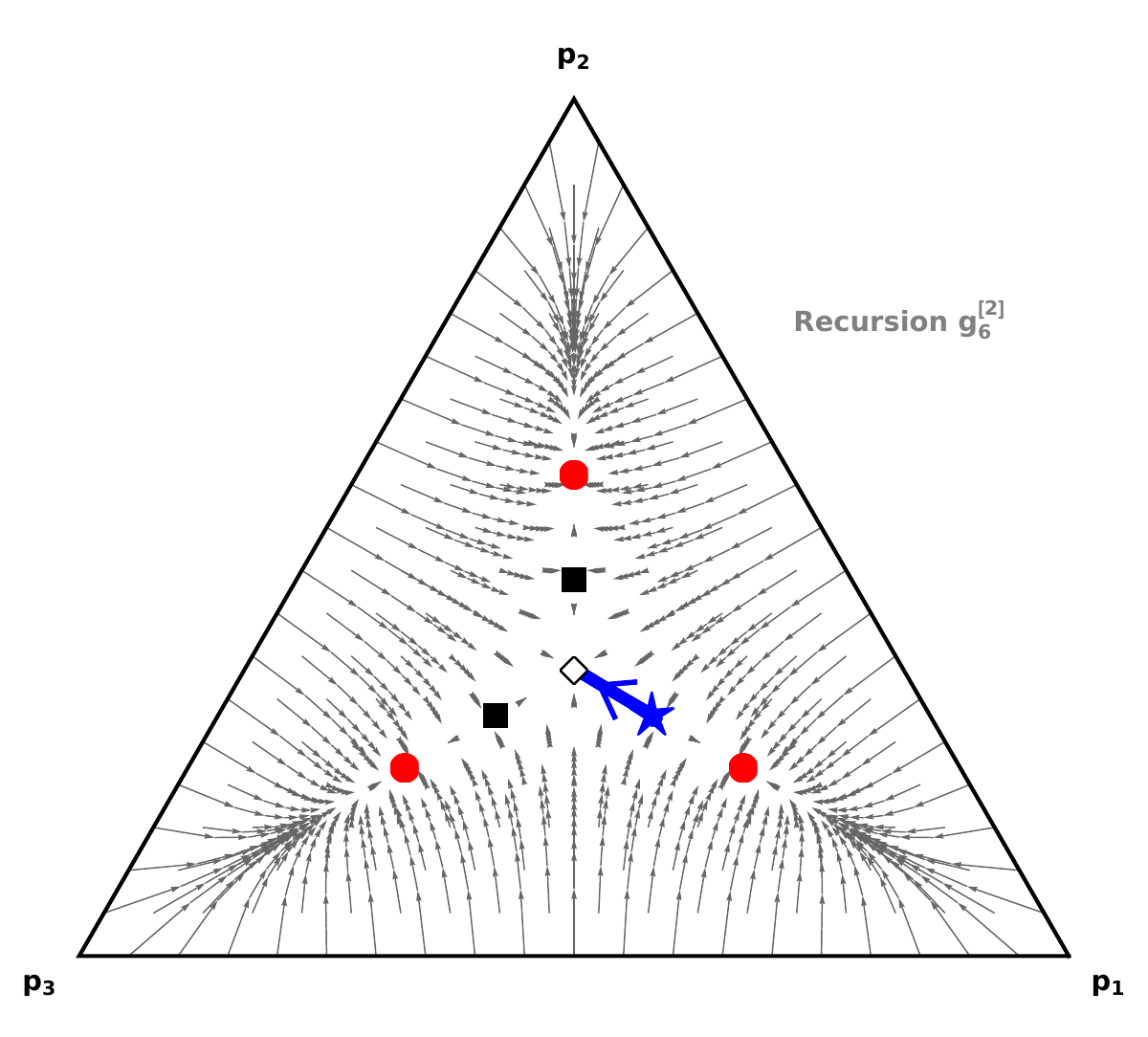}}
    \caption{Depicted are discrete trajectories for two different time-dependent homogeneous recursions on the simplex $\Delta^2$ of the three-state Potts model at $\theta=3.85$ and $t=6$. The figure on the left side shows the trajectories of $g_6^{[1]}$ while the figure on the right side shows the trajectories of $g_6^{[2]}$,both recursions are introduced in \eqref{eq: Homogeneous time-dependent recursion}. The marked points on the simplex correspond again to the fixed points of $f$, as in Figure \ref{fig: Discrete trajectories normalized recursion}, but now for different $\theta$.
     The trajectories of the recursion starting from the saddle point $l_-^{(1)}$ are highlighted with blue thick arrows and consists of $300$ iterations. Selection occurs between time-dependent continuations of $l_{+}^{(1)}$ and $l_{\text{free}}$, whose existence is guaranteed by Lemma \ref{lem: continuation stable fixed points time evolved}.}
    \label{fig: Discrete trajectories normalized recursion small theta}
\end{figure}

\begin{proposition}\label{prop: Orbit time evolved recursion of saddle point}
    Let $t> 0$, $\theta>4$ and $k\in \{2,3\}$. Further, define the sequence $(z_{n}^k(t))_{n\in \N_0}$ recursively by
    \begin{equation*}
        z_0^k(t):=\begin{cases}
            (x_-,1),~&\text{if}~~l=l_-^{(1)},\\
            (1,1),~&\text{if}~~l=l_{\mathrm{free}},
        \end{cases}
    \end{equation*}
    where $x_-$ is defined in Lemma \ref{lem: Fixed points normalized simplex Potts model}, and by $z_{n+1}^k(t):=\widetilde{g}_t^{[k]}(z_{n}^k(t))$ for all $n\in \N_0$. Then, the limit $z^*_k(t):=\lim_{n\rightarrow \infty}z_{n}^k(t)$ exists and is a fixed point of $\widetilde{g}_t^{[k]}$.
\end{proposition}

We will prove the statement for $k=2$. Let $i\in \{1,2,3\}$. We define the \textit{strictly $p_i$-dominant chamber} to be 
\begin{equation}\label{eq: p_i-dominant chamber}
\mathcal{C}_i:=\{p\in\Delta^2: p_i>p_j>0~\forall j\neq i\}.    
\end{equation}
With respect to the metric on $\Delta^2$ induced by the euclidean norm on $\R^3$, the set $\mathcal{C}_i$ coincides with the intersection of the relative interior of $\Delta^2$ and the open Voronoi cell of the $i$-th standard basis vector $e_i\in \Delta^2$ in the Voronoi decomposition generated by the vertices $e_1,e_2,e_3$.
In quotient coordinates, this chamber reads $\widetilde{\mathcal{C}_2}:=(0,1)^2$. The following lemma shows that the $p_2$-dominant chamber is invariant under the recursion.
Thus, once the discrete trajectory enters this chamber, it cannot
leave it.

\begin{lemma}\label{lem: Invariance on p2-dominant chamber}
 Let $t>0$ and $\theta> 4$. The sequence satisfies $\widetilde{g}_t^{[2]}(z_0^2)\in \widetilde{\mathcal{C}_2}$ under the time-dependent recursion and the strictly $p_2$-dominant chamber is invariant under $\widetilde{g}_t^{[2]}$, i.e., we have $\widetilde{g}_t^{[2]}(\widetilde{\mathcal{C}_2})\subset \widetilde{\mathcal{C}_2}$.
 \end{lemma}

\begin{proof}
Let us start with the case $z_0^2(t)=(x_-,1)$. Since $(x_-,1)$ is a fixed point of the time-independent recursion, we have $\widetilde{g}_t^{[2]}(x_-,1)=m_t(x_-,1)$. Note that $m_t\in (0,1)$ for every $t>0$ and $x_-\geq 0$.
A short calculation shows that $x_-\in (0,1)$ if $\theta>4$.  Hence,
the saddle point $(x_-,1)$ is mapped into the strictly $p_2$-dominant chamber after one step, i.e., we have $\widetilde{g}_t^{[2]}(x_-,1)\in\widetilde{\mathcal{C}_2}$ for all $t>0$ and  $\theta> 4$. We now turn to the case $z_0^2(t)=(1,1)$. Since $\widetilde{g}_t^{[2]}(1,1)=m_t(1,1)$, it follows immediately that $\widetilde{g}_t^{[2]}(1,1)\in \widetilde{\mathcal{C}_2}$.

It remains to prove that the chamber $\widetilde{\mathcal{C}}_2$
is forward invariant under $\widetilde{g}_t^{[2]}$. By the definition of $\widetilde{g}_t^{[2]}$ and the fact that $m_t\in (0,1)$ for every $t>0$, one can see directly that $\widetilde{g}_t^{[2]}(x,y)>0$ for all $(x,y)\in (0,1)^2$. It remains to prove that $g_i(x,y)<1$ for all $(x,y)\in (0,1)^2$ and $i\in \{1,2\}$. Note that $g_1(x,y)<1$ is equivalent to $(\theta-1)x<\theta-1$ and $g_2(x,y)<1$ is equivalent to $(\theta-1)y<\theta-1$. Since $\theta>1$, this holds for $x<1$ and $y<1$. Consequently, we have shown that  $\widetilde{g}_t^{[2]}(\widetilde{\mathcal{C}}_2)\subset \widetilde{\mathcal{C}}_2$.  
\end{proof}

By Lemma \ref{lem: Invariance on p2-dominant chamber}, we can restrict ourselves to the analysis of the map $\widetilde{g}^{[2]}_t$ on the strictly $p_2$-dominant chamber $ \widetilde{\mathcal{C}}_2$. In quotient coordinates the 2-dominant chamber 
is the unit square. The dynamics maps any square 
with lower left endpoint and upper right endpoint 
both on the diagonal in these coordinates, to a square of a similar type. 
This is the consequence of the monotonicity 
of the map, in the sense of partial order induced 
by the coordinates. Hence, in order to show that 
the sequence of such squares converges to a singleton on the diagonal, it suffices to consider the two 
sequences of lower left endpoints and upper right endpoints. 
But these two sequences on the diagonal can be shown to converge towards the same 
fixed point on the diagonal, 
by a one-dimensional argument.

\begin{lemma}\label{lem: Orbits in the p2-dominant chamber}
Let $t>0$ and $\theta>4$. Denote by $\Delta := \{ (x,y)\in (0,1)^2 : x=y \}$
the diagonal of the strictly $p_2$-dominant chamber. The map $\widetilde{g}_t^{[2]}$
satisfies $\widetilde{g}_t^{[2]}(\Delta)\subseteq \Delta$ and is componentwise
increasing on $(0,1)^2$, that is,
\begin{equation*}
    \widetilde{g}_t^{[2]}(x,y) \leq \widetilde{g}_t^{[2]}(x',y')
\end{equation*}
componentwise whenever $x,x',y,y'\in(0,1)$, $x\leq x'$, and $y\leq y'$.
Furthermore, let $d_t:(0,1)\to(0,1)$ be the map induced by $\widetilde{g}_t^{[2]}$
on the diagonal, i.e., $ \widetilde{g}_t^{[2]}(x,x) = (d_t(x),d_t(x))$.
If $d_t$ has a unique fixed point $\rho\in (0,1)$, then every discrete trajectory in
$(0,1)^2$ under $\widetilde{g}_t^{[2]}$ converges to $(\rho,\rho)$.
\end{lemma}

\begin{proof}
We start by proving the componentwise monotonicity of $\widetilde{g}_t^{[2]}$. Differentiating gives
\begin{equation}\label{eq: Differential quotient time-dependent FP recursion map}
    D\widetilde{g}_t^{[2]}(x,y)=\frac{2(\theta-1)m_t}{(x+\theta+y)^3}\begin{pmatrix}
        (\theta x+1+y)(\theta+1+y) &  (\theta x+1+y)(1-x)\\
        (x+1+\theta y)(1-y) & (x+1+\theta y)(\theta+x+1)
    \end{pmatrix}.
\end{equation}
Therefore all partial derivatives of $\widetilde{g}_t^{[2]}$ are nonnegative on $(0,1)^2$. Hence $\widetilde{g}_t^{[2]}$ is componentwise increasing.

Let $x\in (0,1)$, then
\begin{equation}\label{eq: Diagonal map d_t}
   \widetilde{g}_{t,2}^{[2]}(x,x)=\widetilde{g}_{t,1}^{[2]}(x,x)=m_t\left(\frac{(\theta+1)x+1}{2x+\theta}\right)^2=d_t(x).
\end{equation}
Thus $\widetilde{g}_t^{[2]}(x,x)=(d_t(x),d_t(x))\in \Delta$ for all $x\in (0,1)$.
Assume that $d_t$ has a unique fixed point $\rho$ on $(0,1)$. The map  $\widetilde{g}_t^{[2]}$ admits a continuous extension to $[0,1]^2$, and its componentwise monotonicity extends to the closed square by continuity. Hence we may also consider the extremal initial conditions $(0,0)$ and $(1,1)$.
Define the sequences $(u_n)_{n\in \N_0}$ and $(v_n)_{n\in \N_0}$ recursively by $u_0=(0,0)$, $v_0=(1,1)$, and $u_{n+1}=\widetilde{g}_t^{[2]}(u_n)$, $v_{n+1}=\widetilde{g}_t^{[2]}(v_n)$ for all $n\in \N_0$. Note that $u_1\in \Delta$ and $v_1\in \Delta$ because equation \eqref{eq: Diagonal map d_t} extends continuously to $x\in [0,1]$ with $d_t(0)=\frac{m_t}{\theta^2}$ and $d_t(1)=m_t$, both of which lie in $(0,1)$. Since the diagonal $\Delta$ is invariant under $\widetilde{g}_t^{[2]}$, we have $u_n=(a_n,a_n)$ and $v_n=(b_n,b_n)$
where $a_{n+1}=d_t(a_n)$ and $b_{n+1}=d_t(b_n),$
with $a_0=0$ and $b_0=1$.
Now let $z_0=(x_0,y_0)\in (0,1)^2$
be arbitrary, and write $z_n=(\widetilde{g}_t^{[2]})^n(z_0)$.
Since $u_0\leq z_0\leq v_0$
componentwise and $\widetilde{g}_t^{[2]}$ is componentwise increasing, we obtain by induction $u_n\leq z_n\leq v_n $
componentwise for every $n\in \N_0$. Therefore, if $z_n=(x_n,y_n)$, each component satisfies $a_n\leq x_n\leq b_n$ and $a_n\leq y_n\leq b_n.$
It remains to justify that the two extremal one-dimensional trajectories $(a_n)_{n\in \N_0}$ and $(b_n)_{n\in \N_0}$ converge
to $\rho$. 

Since $\widetilde{g}_t^{[2]}$ is componentwise increasing and leaves the
diagonal invariant, the induced map $d_t:[0,1]\to[0,1]$ is increasing. We first consider the discrete trajectory starting at $0$. For $a_0=0$, we have $ a_1=d_t(0)=\frac{m_t}{\theta^2}>0=a_0.$
Using that $d_t$ is increasing, it follows inductively that $a_{n+1}=d_t(a_n)\geq d_t(a_{n-1})=a_n$
for every $n\geq 1$. Hence $(a_n)_{n\in\N_0}$ is increasing and bounded
above by $1$. Therefore there exists $\alpha\in(0,1)$ such that
$\lim_{n\rightarrow \infty}a_n=\alpha$. Note that $\alpha\neq 1$ since $d_t(1)=m_t<1$ for all  $t>0$. By continuity of $d_t$, we obtain $\alpha=\lim_{n\to\infty} d_t(a_n)=d_t(\alpha)$.
Thus $\alpha$ is a fixed point of $d_t$. Since $\rho$ is the unique fixed
point of $d_t$ in $(0,1)$, we conclude that $\alpha=\rho$.

Similarly, since $b_0=1$, we have $ b_1=d_t(1)=m_t< 1=b_0.$
Again, using that $d_t$ is increasing, it follows inductively that $(b_n)_{n\in\N_0}$ is decreasing and bounded
below by $0$. Therefore there exists $\beta\in (0,1)$ such that
$\lim_{n\rightarrow \infty}b_n=\beta$. By continuity of $d_t$, we obtain $\beta =\lim_{n\to\infty} d_t(b_n)= d_t(\beta).$
Thus $\beta$ is a fixed point of $d_t$. By uniqueness of the fixed point,
we conclude that $\beta=\rho$ and thus $\lim_{n\rightarrow \infty}a_n=\lim_{n\rightarrow \infty}b_n=\rho$.
Consequently, the components of $z_n$ satisfy $\lim_{n\rightarrow \infty}x_n=\rho$ and $\lim_{n\rightarrow \infty} y_n=\rho$.
Hence $\lim_{n\rightarrow \infty}(\widetilde{g}_t^{[2]})^n(x_0,y_0)=(\rho,\rho)$
for every initial point $(x_0,y_0)\in(0,1)^2$. 

In particular, the fixed point of $\widetilde{g}_t^{[2]}$ in the strictly $p_2$-dominant chamber is unique. Indeed, if $z\in (0,1)^2$ is a fixed point of $\widetilde{g}_t^{[2]}$, then its discrete trajectory is constant. On the other hand, the preceding argument shows that the same trajectory converges to $(\rho,\rho)$. Hence $z=(\rho,\rho)$.
\end{proof}

It remains to prove that $d_t$ possesses a unique fixed point in $(0,1)$. Then the results of Lemma \ref{lem: Invariance on p2-dominant chamber} and Lemma \ref{lem: Orbits in the p2-dominant chamber} imply Proposition \ref{prop: Orbit time evolved recursion of saddle point}.

\begin{lemma}
Let $\theta> 4$. The map $d_t:[0,1]\rightarrow [0,1]$ from Lemma \ref{lem: Orbits in the p2-dominant chamber}, defined in \eqref{eq: Diagonal map d_t}, has a unique fixed point $\rho$ in $(0,1)$.     
\end{lemma}

\begin{proof}  
Since $d_t(0)=m_t/\theta^2>0$ and $d_t(1)=m_t<1$,
the points $x\in \{0,1\}$ are not fixed points. Hence every fixed point lies in
$(0,1)$. Define 
\begin{equation*}
    h(x):=\frac{((\theta+1)x+1)^2}{x(2x+\theta)^2}
\end{equation*}
for $x\in (0,1)$. Thus, the fixed points $x$ of $d_t$ in $(0,1)$ are precisely the solutions of $h(x)=\frac{1}{m_t}$.
We first note that $\lim_{x\downarrow 0}h(x)=+\infty,$
while $ h(1)=1$.
Since $m_t<1$, we have $\frac{1}{m_t}>1$. Hence, by continuity of $h$, the intermediate value theorem implies that the equation has at least one solution in $(0,1)$.

We now study the monotonicity of $h$. Its logarithmic derivative is
\begin{equation*}
    \frac{h'(x)}{h(x)}=-\frac{2(\theta+1)x^2-(\theta^2+\theta-6)x+\theta}{x(2x+\theta)((\theta+1)x+1)}.
\end{equation*}
The denominator on the right hand side and $h(x)$ are strictly positive on $x\in(0,1]$. Hence the sign of
$h'(x)$ is the opposite of the sign of
\begin{equation*}
    p_\theta(x)
:=
2(\theta+1)x^2-(\theta^2+\theta-6)x+\theta.
\end{equation*}
If $\theta> 4$, we have
$p_\theta(0)=\theta>0,$
and $p_\theta(1)=
-(\theta-4)(\theta+2)
< 0$. Hence, by the intermediate value theorem, $p_\theta$ has at least one root in $(0,1)$. Since $p_\theta(1)<0$ and $\lim_{x\to\infty}p_\theta(x)=\infty$, it has another root in $(1,\infty)$. Thus for every $\theta> 4$, $p_\theta$ has two nonnegative roots, one in $(0,1)$ and the other in $(1,\infty)$.

Let $\alpha\in(0,1)$ be the root of $p_\theta$ in $(0,1)$. From the sign of $p_\theta$ it follows that $h'(x)<0$ for $x\in(0,\alpha)$ and $h'(x)>0$ for $x\in(\alpha,1)$. Moreover, the function $h$ satisfies $\lim_{x\downarrow 0}h(x)=+\infty$ and $h(\alpha)\leq h(1)=1<\frac{1}{m_t}$.
Thus the equation $h(x)=m_t$ has exactly one solution in $(0,\alpha)$ and no solution in $[\alpha,1]$. Equivalently, the fixed point equation $x=d_t(x)$ has a unique solution in $(0,1)$.
\end{proof}

Consequently, part $a)$ of Proposition \ref{prop: Total Badness q=3 Potts model} follows.

\subsection{Existence of sets invariant under all perturbed recursions} \label{sec: observation-invariant-set}

In this subsection, we will prove part b) of Proposition \ref{prop: Total Badness q=3 Potts model}. A direct application of Proposition \ref{prop: time-dependent map: tree contraction} yields that each unperturbed $f$-stable fixed point $l_+^{(i)}$ where $i\in \{2,3\}$ admits a local trapping region that, for all sufficiently large $t$, is uniform over the family of time-dependent inhomogeneous recursions $\left(G_t^{[k]}\right)_{k\in \{1,2,3\}}$.

\begin{corollary}\label{cor: differential normalized potts map}
Assume that $\theta>4$ and fix $i\in \{2,3\}$. Then there exists $R_i>0$ such that, for every $r_i\in (0,R_i)$, there exists a time $t^{(i)}_1=t^{(i)}_1(r_i)<\infty$ such that $G_t^{[k]}\left((U_{r_i})^2\right)\subset U_{r_i}$  for all $t\geq t^{(i)}_1$ and all $k\in S$. Here, the set $U_{r_i}=U_{r_i}(l_+^{(i)})$ denotes the neighborhood of $l_+^{(i)}$ defined in Proposition \ref{prop: time-dependent map: tree contraction}.
\end{corollary}

 We then prove that the $f$-stable fixed points $l_+^{(2)}$ and $l_+^{(3)}$ admit smooth time-dependent continuations that are fixed points of $g_t^{[2]}$ and $g_t^{[3]}$, respectively, and identify the continuations with the limits obtained in part a).

\begin{lemma}\label{lem: continuation stable fixed points time evolved}
Let $i\in \{2,3\}$ and $R_i>0$ the threshold given in Corollary \ref{cor: differential normalized potts map}. Then, for each $r_i\in (0,R_i)$ there exists $t^{(i)}_2=t^{(i)}_2(r_i)<\infty$ 
and a unique family $\left(l^{(i)}(t)\right)_{t\geq t^{(i)}_2}$ 
such that the following holds. 

The family defines a continuation of $l^{(i)}_+$, i.e. it satisfies $\lim_{t\rightarrow \infty}l^{(i)}(t)=l^{(i)}_+$ and $l^{(i)}(t)\in U_{r_i}$ for all $t\geq t^{(i)}_2$. Here, $l^{(i)}(t)$ are the limits of Proposition \ref{prop: Total Badness q=3 Potts model} part a) and $U_{r_i}$ the local trapping region of Corollary \ref{cor: differential normalized potts map}.
\end{lemma}

\begin{proof}
The proof of the existence of a non-degenerate continuation follows from the implicit 
function theorem. Here the main point, i.e. its applicability 
will follow from making use of the assumption of the $f$-stability of the unperturbed fixed point considered. In more detail, in our situation we proceed as follows. 
Again, it suffices to prove the statement for $i=2$, since the case $i=3$ follows by symmetry. Fix $r_2\in (0,R_2)$. We first construct a local continuation of the unperturbed stable fixed point $l_+^{(2)}$ as a fixed point of the time-dependent homogeneous recursion.

Let $k\in \{1,2,3\}$. Note that we can rewrite the $k$-th entry of the time-dependent homogeneous recursion as follows
\begin{equation*}
    \left(g_t^{[2]}(l_1,l_2,l_3)\right)_k=\frac{m_t^{(1-\delta_{k2})}\left(\sum_{j=1}^3 Q(k,j)l_j\right)^2}{\sum^3_{i=1}m_t^{(1-\delta_{i2})}\left(\sum_{j=1}^3 Q(i,j)l_j\right)^2}
\end{equation*}
where we recall that $m_t=\frac{1-e^{-\frac{3}{2} t}}{1+2e^{-\frac{3}{2} t}}$ which coincides with $\frac{P_t(k,2)}{P_t(2,2)}$ for $k\in \{1,3\}$. Therefore, after replacing the parameter $m_t$ in the definition of $g_t^{[2]}$ by an independent variable $m$, we may regard the time-dependent homogeneous recursion as a smooth family $g_m^{[2]}:\Delta^2\rightarrow \Delta^2$. At $m=1$, this map coincides with the unperturbed homogeneous boundary law recursion $f$.

We use the global coordinates on the relative interior of the simplex. In more detail, define $\psi: B\rightarrow \mathrm{ri}(\Delta^2)$ by $\psi(x,y):=(x,y,1-x-y)$ where $B:=\{(x,y)\in (0,\infty)^2:~x+y<1\}$ and $\mathrm{ri}(\Delta^2)$ was defined in \eqref{eq: Relative interior Simplex}. Further, we introduce the map $\hat{g}_m^{[2]}:=\psi^{-1}\circ g_m^{[2]}\circ  \psi$ and $\Phi^{[2]}(m,x,y):=\hat{g}_m^{[2]}(x,y)-(x,y)$. Then $\Phi_m^{[2]}$ is $C^\infty$ (infinitely differentiable) in a neighborhood of $\left(1,\psi^{-1}(l_+^{(2)})\right)$. Moreover, it satisfies $\Phi^{[2]}\left(1,\psi^{-1}(l_+^{(2)})\right)=0$ because $l_+^{(2)}$ is a fixed point of the unperturbed map $f$. Taking the derivative with respect to $(x,y)$, which we denote by $D_{(x,y)}$, results in $D_{(x,y)}\Phi^{[2]}(m,x,y)=D\hat{g}_m^{[2]}(x,y)-I_2$. At the point $m=1$, we obtain $\hat{g}_1=\hat{f}:=\psi^{-1}\circ f\circ \psi$, hence  $D_{(x,y)}\Phi^{[2]}\left(1,\psi^{-1}(l_+^{(2)})\right)=D\hat{f}(\psi^{-1}(l_+^{(2)}))-I_2$. In order to apply the implicit function theorem, this matrix has to be invertible.  Set $p:=\psi^{-1}(l_+^{(2)})$. By the chain rule and since $l_+^{(2)}$ is a fixed point of $f$, we obtain 
\begin{equation*}
    D\hat{f}(p)=D\psi^{-1}\left(f(\psi(p))\right)\circ Df\left(\psi(p)\right)\circ D\psi(p)=D\psi^{-1}(l_+^{(2)})\circ Df(l_+^{(2)})\circ D\psi(p).
\end{equation*}
Note that $D\psi(p):\R^2\rightarrow T\Delta^2$ is defined by $D\psi(p)(a,b)=(a,b,-a-b)$ and $D\psi^{-1}(l_+^{(2)}):T\Delta^2\rightarrow \R^2$ is defined by $D\psi^{-1}(p)(\widetilde{a},\widetilde{b},\widetilde{c})=(\widetilde{a},\widetilde{b})$. These maps are inverse to each other and we can regard $D\psi(p):\R^2\rightarrow T\Delta^2$ as an isomorphism. Consequently, we obtain 
\begin{equation*}
      D\hat{f}(p)=\left(D\psi(p)\right)^{-1}\circ Df(l_+^{(2)})\circ D\psi(p)
\end{equation*}
and hence, the maps $ D\hat{f}(p)$ and $Df(l_+^{(2)})$ are conjugate linear maps and, in particular, have the same eigenvalues. By Lemma \ref{lem: Fixed points normalized simplex Potts model}, the matrix $Df(l_+^{(2)})$ possesses a spectral radius which is smaller than one and hence, the matrix $ D\hat{f}(p)-I_2$ is invertible. 

The implicit function theorem yields the existence of 
an open interval $I\subset (0,\infty)$ containing $1$ and a $C^\infty$-map $m\mapsto \hat{l}_m^{(2)}\in B$ such that $\hat{l}_1^{(2)}=\psi^{-1}(l_+^{(2)})$ and $\hat{g}_m^{(2)}(\hat{l}_m^{(2)})=\hat{l}_m^{(2)}$ for all $m\in I$. Setting $l_m^{(2)}:=\psi(\hat{l}_m^{(2)})$, we obtain a $C^\infty$-continuation of $l_+^{(2)}$ satisfying $g_m^{(2)}(l_m^{(2)})=l_m^{(2)}$ for all $m\in I$ and $l_1^{(2)}=l_+^{(2)}$. Since $\lim_{t\rightarrow \infty}m_t=1$ there exists $t^{(2)}_I<\infty$ such that $m_t\in I$ for all $t\geq t^{(2)}_I$. Moreover, by continuity of the map $m\mapsto l_m^{(2)}$, we have $\lim_{t\rightarrow \infty}l_{m_t}^{(2)}=l_+^{(2)}$. Hence, there exists $t^{(2)}_U<\infty$ such that $l_{m_t}^{(2)}\in U_{r_2}$ for all $t\geq t^{(2)}_U$.

It remains to identify this local branch of fixed points with the limiting fixed points $l^{(2)}(t)$ obtained in part $a)$. By part $a)$, the point $l^{(2)}(t)$ lies in the strictly $p_2$-dominant chamber $\mathcal{C}_2$ for every $t>0$. Moreover, the fixed point of $g_t^{[2]}$ in $\mathcal{C}_2$ is unique, see Lemma \ref{lem: Invariance on p2-dominant chamber}. Since $\lim_{t\rightarrow \infty}m_t=1$ and $\lim_{m\rightarrow 1}l_{m}^{(2)}=l^{(2)}_+$, there exists $t^{(2)}_C<\infty$ such that $l_{m_t}^{(2)}\in \mathcal{C}_2$ for all $t\geq t^{(2)}_C$. Thus, for every $t\geq t^{(2)}_C$, both $l_{m_t}^{(2)}$ and $l^{(2)}(t)$ are fixed points of $g_t^{[2]}$ in $\mathcal{C}_2$. By uniqueness $l^{(2)}(t)=l_{m_t}^{(2)}$ for all $t\geq t^{(2)}_C$.

Therefore, setting $t^{(2)}_2:=\max\left\{t^{(2)}_I,t^{(2)}_U,t^{(2)}_C\right\}$ proves the claim. 
\end{proof}

Consequently, part b) of Proposition \ref{prop: Total Badness q=3 Potts model} follows by choosing $U_i:=U_{r_i}(l_+^{(i)})$ for arbitrary $r_i\in (0,R_i)$, where $i\in \{2,3\}$, and by taking $t\geq \max_{\substack{i\in \{2,3\}\\ j\in \{1,2\}}}t^{(i)}_{j}(r_i)$.

\begin{remark}\label{rk: Radius trapping regions total badness}
    It follows from the proof of Proposition \ref{prop: time-dependent map: tree contraction} and part b) of Proposition \ref{prop: Total Badness q=3 Potts model} that the radii defining the neighborhoods $U_2$ and $U_3$ may be chosen arbitrarily small, provided that the time threshold is increased if necessary. We will use this freedom in part c) to choose the neighborhoods sufficiently small so that the relative $H_2$-weights, see below \eqref{eq: Rewrite R_{t,2}}, appearing in the root marginals remain uniformly close to their values at the corresponding $f$-stable fixed points. 
\end{remark}

\subsection{Selectability and uniform separation of the root probabilities}
\label{sec: observation-root-separation}

Finally, we prove part c) of Proposition \ref{prop: Total Badness q=3 Potts model} by verifying that the two reference root marginals are indeed
distinct. For this purpose, we note that $P_t(1,2)=P_t(3,2)$ for every finite $t>0$, compare \eqref{eq: time-dependent spin flip}. 
Furthermore, one can see that
\begin{align}\label{eq: Rewrite R_{t,2}}
R_{t,2}(l^{(1)},l^{(2)},l^{(3)})=P_t(1,2)+\Big(P_t(2,2)-P_t(1,2)\Big)W_2(l^{(1)},l^{(2)},l^{(3)})
\end{align}
where $W_2(l^{(1)},l^{(2)},l^{(3)}):=\frac{H_2(l^{(1)},l^{(2)},l^{(3)})}{\sum_{a=1}^3 H_a(l^{(1)},l^{(2)},l^{(3)})}$ denotes the \textit{relative $H_2$-weight}, and the functions $H_a$ are defined in \eqref{eq: Definition R_t,k and H_a(l)}.
Since $(P_t(2,2)-P_t(1,2))=\mathrm{exp}\left(-\frac{3t}{2}\right)>0$, it suffices to compare these relative $H_2$-weights. We first evaluate them at the unperturbed $f$-stable fixed points $l^{(2)}_+$ and $l^{(3)}_+$.
For equal incoming messages $l\in \Delta^2$, one has $H_a(l,l,l)=\left(\sum_{b=1}^3 Q(a,b)l(b)\right)^{3}.$
Set $A:=(2+\theta x_+)^{3}$ and $B:=(\theta+x_++1)^{3}$.
The relative $H_2$-weights satisfy
\begin{equation*}
W_2(l_+^{(2)},l_+^{(2)},l_+^{(2)})=\frac{A}{A+2B}>\frac{B}{A+2B}=W_2(l_+^{(3)},l_+^{(3)},l_+^{(3)})
\end{equation*}
This inequality is equivalent to $(\theta-1)(x_+-1)>0$ which holds since  $\theta>4$ and $x_+>1$.
By the fixed point continuation of Lemma \ref{lem: continuation stable fixed points time evolved} and the identification with the limits from part a), we have $\lim_{t\rightarrow \infty}l^{(i)}(t)=l_+^{(i)}$ for each $i\in \{2,3\}$.
Furthermore, by definition, the map $W_{2}$ is continuous on a neighborhood of the relative interior of $(\Delta^2)^3$, since all entries of $Q$, as well as all boundary messages under consideration, are strictly positive. In particular, the
denominator is bounded away from zero on compact neighborhoods of the
$f$-stable fixed points.
Therefore, there exists $t_3<\infty$ such that 
\begin{equation*}
\left|W_{2}\left(l^{(2)}(t),l^{(2)}(t),l^{(2)}(t)\right)-W_{2}\left(l^{(3)}(t),l^{(3)}(t),l^{(3)}(t)\right)\right|>\frac{\delta}{2}
\end{equation*}
for all $t\geq t_3$ where $\delta:=\frac{A-B}{(A+2B)}>0$. Note that, by reading the proof of Lemma \ref{lem: Upper bound root marginals}, the upper bound in \eqref{eq: Bound difference R_t,k} also holds for the difference of two relative $H_2$-weights. Consequently, we obtain 
\begin{equation*}
    \left|W_2(l^{(1)},l^{(2)},l^{(3)})-W_2(\widetilde{l}^{(1)},\widetilde{l}^{(2)},\widetilde{l}^{(3)})\right|\leq 2\theta^3 \sum^3_{i=1}\|l^{(i)}-\widetilde{l}^{(i)}\|_1
\end{equation*}
for all $l^{(1)},l^{(2)},l^{(3)}\in \Delta^2$ and $\widetilde{l}^{(1)},\widetilde{l}^{(2)},\widetilde{l}^{(3)}\in \Delta^2$. By choosing the radii of the trapping neighborhoods $U_2$ and $U_3$ from part b) sufficiently small, as explained in Remark \ref{rk: Radius trapping regions total badness}, we may ensure that, for each $i\in \{2,3\}$ and every $t\geq \max\{t^{(i)}_1,t^{(i)}_2\}$ the estimate
\begin{equation*}
\left|W_{2}\left(l^{(1)},l^{(2)},l^{(3)}\right)-W_{2}\left(l^{(i)}(t),l^{(i)}(t),l^{(i)}(t)\right)\right|<\frac{\delta}{8}
\end{equation*}
holds for all $l^{(1)},l^{(2)},l^{(3)}\in U_i$. Here, the time-thresholds $t^{(i)}_1$ and $t^{(i)}_2$ are provided by Corollary \ref{cor: differential normalized potts map} and Lemma \ref{lem: continuation stable fixed points time evolved}, respectively.
Therefore, for all $l^{(1)},l^{(2)},l^{(3)}\in U_2$ and all
$\widetilde{l}^{(1)},\widetilde{l}^{(2)},\widetilde{l}^{(3)}\in U_3$, the reverse triangle inequality gives
\begin{align*}
&
\left|R_{t,2}(l^{(1)},l^{(2)},l^{(3)})-R_{t,2}(\widetilde{l}^{(1)},\widetilde{l}^{(2)},\widetilde{l}^{(3)})\right|>\frac{\delta}{4}\cdot\mathrm{exp}\left(-\frac{3t}{2}\right).
\end{align*}
Choosing $C:=\frac{\delta}{4}>0$, we obtain the desired estimate and thereby prove part $c)$ of Proposition \ref{prop: Total Badness q=3 Potts model} with $t_B:=\max\{t^{(2)}_1,t^{(3)}_1,t^{(2)}_2,t^{(3)}_2,t_3\}$.

\section*{Appendix A: Useful generalities about tree-indexed families of contractions}
\addcontentsline{toc}{section}{Appendix A: Useful generalities about tree-indexed families of contractions} \label{Appendix: Generalities about tree-indexed contractions}

We expand some definitions for discrete dynamical systems
indexed on $\N$ to the case where the index set is a Cayley tree of order $s$ with root $o\in V$.
This is useful for boundary-law iterations with non-homogeneous configurations. For a rooted $s$-ary tree, we write
\begin{equation*}
    x^{(i)}=\left(x^{(i)}_1,\dots,x^{(i)}_{s^i}\right)\in U^{s^i}
\end{equation*}
for the collection of values on the $i$-th annulus. Here, $i\in\N_0$ denotes the distance from the root and $j=1,\dots,s^i$ enumerates the vertices in that annulus. A sequence $\left(x^{(i)}\right)_{i\in\N_0}$
is then called a \textit{tree-indexed sequence}. Let $(G_a)_{a\in A}$ be a family of $s$-tree contractions, see Definition \ref{def: s-tree contractions}, and $(G^{(i)})_{i\in \N_0}$ where $G^{(i)}=(G^{(i)}_1,\dots,G^{(i)}_{s^i})$ for every $i\in \N_0$
be a \textit{tree-indexed sequence of maps} with $G^{(i)}_j\in \{G_a:a\in A\}$. Further, let $n\in\N$ and $u^{(n)}=(u^{(n)}_1,\dots,u^{(n)}_{s^n})\in U^{s^n}$ be \textit{depth-n boundary data}. 

In the following, we formally define the depth-$n$ recursive application of the family $(G^{(i)})_{i\in \N_0}$ to $u^{(n)}$.
First, define $X^{(n,n)}_j := u^{(n)}_j$ for $j=1,\dots,s^n$. We then define recursively
\begin{equation}\label{eq: Depth-n tree composition}
      X^{(n,i)}_j:=G^{(i)}_j\left(X^{(n,i+1)}_{(j-1)s+1},\dots,X^{(n,i+1)}_{js}\right)
\end{equation}
for $i=n-1,\dots,0$ and $j=1,\dots,s^i$. The quantity $X^{(n,i)}_j$ denotes the value associated with the $j$-th vertex at level $i$ of the depth-$n$ $s$-ary tree. It is obtained by recursively evaluating the subtree rooted at this node recursively under the family of $s$-tree contractions $(G_a)_{a\in A}$, starting from the boundary data $u^{(n)}$ prescribed at level $n$. 
Then,
we denote by $x_n:=X^{(n,0)}_1$
the \textit{output of the depth-$n$ tree composition} with boundary data $u^{(n)}$.

\begin{lemma}\label{lem: decay influence of inhomogeneous tree contractions}
Let $(U,d)$ be a metric space, $(G_a)_{a\in A}$ a family of $s$-tree contractions with corresponding contraction constants $\lambda_a$ satisfying
\begin{equation*}
  \lambda:=\sup_{a\in A}\lambda_a<1.
\end{equation*}
Furthermore, let $U_0\subset U$ satisfy $D:=\mathrm{diam}(U_0)=\sup_{x,y\in U_0}d(x,y)<\infty$
and $G_a\big((U_0)^s\big)\subset U_0$ for every $a\in A$.
Then, large distance behavior of any tree-indexed sequence of initial boundary data $U_0$ under any tree-indexed composition of maps $(G_a)_{a\in A}$ becomes asymptotically independent of the intial boundary data $U_0$. 

More precisely, let $(G^{(i)})_{i\in\N_0}$
be a tree-indexed sequence of maps from the family $(G_a)_{a\in A}$, where $G^{(i)}_j$
is attached to the $j$-th vertex in the $i$-th annulus and $G_j^{(i)}=G_{a(i,j)}$ for arbitrary $a(i,j)\in A$. Let
\begin{equation*}
    u^{(n)}=\left(u^{(n)}_1,\dots,u^{(n)}_{s^n}\right),~~\text{and}~~v^{(n)}=\left(v^{(n)}_1,\dots,v^{(n)}_{s^n}\right)
\end{equation*}
be arbitrary boundary data in $U_0^{s^n}$ for each $n\in\N$.
Let $x_n$ and $y_n$ denote the output of the same depth-$n$ tree composition with boundary data
$u^{(n)}$ and $v^{(n)}$, respectively.
Then
\begin{equation*}
     d(x_n,y_n)\leq \lambda^n D
\end{equation*}
for all $n\in \N$.
\end{lemma}

\begin{proof}
For $n\in\N$, define the elements appearing in the two
tree evaluations recursively. Let us denote by $(X^{(n,i)})_{i\in \{0,\ldots,n\}}$ and $(Y^{(n,i)})_{i\in \{0,\ldots,n\}}$ the values of the depth-n tree compositions with boundary data $u^{(n)}$ and $v^{(n)}$, respectively. Here, for each $i\in \{0,\ldots,n\}$, the vectors
\begin{equation*}
    X^{(n,i)}:=(X^{(n,i)}_1,\ldots,X^{(n,i)}_{s^i})~~\text{and}~~Y^{(n,i)}:=(X^{(n,i)}_1,\ldots,X^{(n,i)}_{s^i})
\end{equation*}
collect the values at depth $i$, as defined in \eqref{eq: Depth-n tree composition}.

Since $u^{(n)}_j,v^{(n)}_j\in U_0$ and
$G_a((U_0)^s)\subset U_0$ for every $a\in A$, all elements
$X^{(n,l)}_j$ and $Y^{(n,l)}_j$ belong to $U_0$. In particular,
at the leaves we have
\begin{equation*}
    d(X^{(n,n)}_j,Y^{(n,n)}_j)\leq \mathrm{diam}(U_0)=D
\end{equation*}
for $j=1,\ldots,s^n$. Now set
\begin{equation*}
    \Delta_l:=\max_{1\leq j\leq s^l}d(X^{(n,l)}_j,Y^{(n,l)}_j)
\end{equation*}
for all $l\in \{0,\ldots,n\}$. Since each $G^{(l)}_j$ is an $s$-tree contraction with contraction
constant at most $\lambda$, we obtain
\begin{equation*}
    d(X^{(n,l)}_j,Y^{(n,l)}_j)\leq \lambda \max_{1\leq r\leq s} d\Big(X^{(n,l+1)}_{(j-1)s+r},Y^{(n,l+1)}_{(j-1)s+r}\Big).
\end{equation*}
Hence $ \Delta_l\leq \lambda \Delta_{l+1}$
for every $l=0,\dots,n-1$. Since $\Delta_n\leq D$, induction gives $\Delta_0\leq \lambda^n D$.
Therefore
\begin{equation*}
    d(x_n,y_n)=d(X^{(n,0)}_1,Y^{(n,0)}_1)\leq\lambda^n D.
\end{equation*}
\end{proof}

\section*{Appendix B: Lipschitz bound for root conditional probabilities}
\addcontentsline{toc}{section}{Appendix B:  Lipschitz bound for root conditional probabilities} \label{Appendix: Lipschitz bound for root marginals}

The following lemma provides a Lipschitz estimate for the dependence of the root conditional probabilities on the incoming boundary messages. 

\begin{lemma}\label{lem: Upper bound root marginals}
    Let $l,\widetilde{l}\in \left(\Delta^{q-1}\right)^{d+1}$ with $l=(l^{(1)},\dots,l^{(d+1)})$ and $\widetilde{l}=(\widetilde{l}^{(1)},\dots,\widetilde{l}^{(d+1)})$. Define 
    \begin{equation}\label{eq: Definition R_t,k and H_a(l)}
R_{t,k}(l)
:=\frac{1}{Z(l)}\sum_{a\in S} P_t(a,k)H_a(l)~~\text{where}~~H_a(l)
:=\prod_{m=1}^{d+1}\sum_{b\in S} Q(a,b)l^{(m)}(b)
\end{equation}
and $Z(l):=\sum_{a\in S}H_a(l)$ for all $t>0$ and $k,a \in S$. Then, the following estimate in terms of $l$ and $\widetilde{l}$ holds for the root marginals 
\begin{equation}\label{eq: Bound difference R_t,k}
    |R_{t,k}(l)-R_{t,k}(\widetilde{l})|\leq 2\left(\frac{\max_{i,j\in S}Q(i,j)}{\min_{i,j\in S}Q(i,j)}\right)^{d+1}\sum^{d+1}_{i=1}\|l^{(i)}-\widetilde{l}^{(i)}\|_{1}.
\end{equation}
\end{lemma}

\begin{proof}
   Note that $P_t(1,1)=P_t(i,i)$ and $P_t(1,2)=P_t(i,j)$ for all $i,j\in S$ with $i\neq j$, compare the definition of $P_t$ in \eqref{eq: time-dependent spin flip}. 
Therefore, we have
\begin{align*}
R_{t,k}(l)=P_t(1,2)+\Big(P_t(1,1)-P_t(1,2)\Big)\frac{H_k(l)}{Z(l)}.
\end{align*}
Thus, since $P_t(1,1)-P_t(1,2)=\text{exp}\big(-\frac{q}{q-1}t\big)\in (0,1)$ for all $t>0$, it suffices to compare the relative $H_k$-weights. Indeed, we have
\begin{equation*}\label{eq: Bound on difference R_t,k proof}
    |R_{t,k}(l)-R_{t,k}(\widetilde{l})|\leq \Bigg|\frac{H_k(l)}{Z(l)}-\frac{H_k(\widetilde{l})}{Z(\widetilde{l})}\Bigg|.
\end{equation*}
Introduce the notions $x_{a,m}:=\sum_{b\in S} Q(a,b)l^{(m)}(b)$ and $x'_{a,m}:=\sum_{b\in S} Q(a,b)\widetilde{l}^{(m)}(b)$
for all $a\in S$ and $m\in \{1,\dots,d+1\}$.
Then, we have $H_a(l)=\prod^{d+1}_{m=1}x_{a,m}$ and $H_a(\widetilde{l})=\prod^{d+1}_{m=1}x'_{a,m}$. Since $l,\widetilde{l}\in \left(\Delta^{q-1}\right)^{d+1}$, the following inequalities hold true 
    \begin{equation}\label{eq: Bounds on x_a,m}
        \min_{i,j\in S} Q(i,j)\leq x_{a,m}  \leq \max_{i,j\in S}Q(i,j)~~\text{and}~~\min_{i,j\in S} Q(i,j)\leq x'_{a,m}  \leq \max_{i,j\in S}Q(i,j)
    \end{equation}
    for all $a\in S$ and $m\in \{1,\dots,d+1\}$. Furthermore, we have 
    \begin{equation}\label{eq: Bounds on Difference x_a,m}
        |x_{a,m}-x'_{a,m}|\leq \max_{i,j\in S}Q(i,j)\|l^{(m)}-\widetilde{l}^{(m)}\|_1
    \end{equation}
    for all $m\in \{1,\dots,d+1\}$. Define $F_k:=\left(\prod^k_{m=1}x_m\right)\left(\prod^{d+1}_{m=k+1}x'_m\right)$ for $k\in \{0,\dots,d+1\}$.
    A telescoping estimate yields
    \begin{equation*}
        \left|\prod^{d+1}_{m=1} x_{a,m}-\prod^{d+1}_{m=1} x'_{a,m}\right|=\left|\sum^{d+1}_{k=1} (F_k-F_{k-1})\right|\leq \sum^{d+1}_{k=1}\left(\prod_{m=1}^{k-1}x_{a,m}\right)|x_{a,k}-x'_{a,k}|\left(\prod_{m=k+1}^{d+1} x'_{a,m}\right).
    \end{equation*}
    Using \eqref{eq: Bounds on x_a,m} and \eqref{eq: Bounds on Difference x_a,m} in the above bound, results in
    \begin{equation}\label{eq: Bound difference H_a}
        |H_a(l)-H_a(\widetilde{l})|\leq \max_{i,j\in S}Q(i,j)^{d+1}\sum^{d+1}_{m=1}\|l^{(m)}-\widetilde{l}^{(m)}\|_1.
    \end{equation}
    The triangle inequality yields
    \begin{equation}\label{eq: Bound on difference H_a/Z}
        \left| \frac{H_a(l)}{Z(l)}-\frac{H_a(\widetilde{l})}{Z(\widetilde{l})}\right|\leq \frac{|H_a(l)-H_a(\widetilde{l})|}{Z(l)}+\frac{H_a(\widetilde{l})}{Z(l)Z(\widetilde{l})}|Z(l)-Z(\widetilde{l})|\leq \frac{2\sum_{a\in S}|H_a(l)-H_a(\widetilde{l})|}{Z(l)}.
    \end{equation}
    Here, we used $H_a(\widetilde{l})\leq Z(\widetilde{l})$ and $|Z(l)-Z(\widetilde{l})|\leq \sum_{a\in S}|H_a(l)-H_a(\widetilde{l})|$ to obtain the second inequality. From \eqref{eq: Bounds on x_a,m}, we also obtain $Z(l)=\sum_{a\in S}H_a(l)\geq q\min_{i,j\in S} Q(i,j)^{d+1}$. Applying the lower bound for $Z$ to the denominator in \eqref{eq: Bound on difference H_a/Z} and the upper bound \eqref{eq: Bound difference H_a} to the numerator yields the desired inequality \eqref{eq: Bound difference R_t,k}.
\end{proof}

\section*{Appendix C: Background on fixed points for the Potts model}
\addcontentsline{toc}{section}{Appendix C:  Background on fixed points for the Potts model} \label{Appendix: Computation BL's and stability}

In the following, we analyze the existence and stability of fixed points for the time-independent normalized homogeneous boundary-law recursion $f$ for the three-state Potts model, given in \eqref{eq: Normalized homogeneous BL recursion Potts}, where the stability is characterized by the spectral radius of the differential at the corresponding fixed point. For more general results concerning the Potts model on trees, see \cite{Ro13}.

\begin{proof}[Proof of Lemma \ref{lem: Fixed points normalized simplex Potts model}]
We first explain how the fixed points appearing in the lemma are obtained.
Since the Potts interaction is invariant under permutations of the spin values,
the boundary-law equations are equivariant under permutations of the
three components of $l=(l_1,l_2,l_3)$. 
Thus the stability classification obtained for $l$ also holds for permutations of the components. 

Moreover, every positive spatially homogeneous boundary law has at least two
equal components. Indeed, since the boundary-law equation is imposed only up to
multiplicative normalization, there exists a constant $c>0$ such that $\sqrt{l_i}
=
c \sum_{j=1}^3 Q(i,j)l_j$ for all $i\in \{1,2,3\}$.
Using $Q(i,i)=\theta$ while $Q(i,j)=1$ for $i\neq j$,
we obtain, for $i\neq k$,
\begin{equation*}
\sqrt{l_i}-\sqrt{l_k}=c(\theta-1)(l_i-l_k)=c(\theta-1)(\sqrt{l_i}-\sqrt{l_k})(\sqrt{l_i}+\sqrt{l_k}).
\end{equation*}
Hence, if $l_i\neq l_k$, then $\sqrt{l_i}+\sqrt{l_k}=\frac{1}{c(\theta-1)}$.
It is impossible for this to hold for all three pairs with $l_1,l_2,l_3$
pairwise distinct, since then the three sums $\sqrt{l_i}+\sqrt{l_j}$ for $i\neq j$ would have to be equal. However, this immediately implies that two of the $l_i$ must be equal. Thus at least two components of every positive homogeneous
boundary law coincide. 

By permutation symmetry, it is enough to consider the case $l_2=l_3$.
Equivalently, we may consider a boundary law proportional to $l=(x,1,1)$,
and then normalize it on the simplex. We write $x=u^2$ with $u>0$. Taking the ratio of the first and second coordinates in the boundary-law equation, the resulting equation is equivalent to
$u^3-\theta u^2+(\theta+1)u-2=0.$
This polynomial has the trivial root $u=1$, corresponding to the free fixed
point $\left(\frac{1}{3},\frac{1}{3},\frac{1}{3}\right)$.
The nontrivial fixed points are obtained from $u^2+(1-\theta)u+2=0.$
Solving this equation gives
\begin{equation*}
u_\pm=\frac{\theta-1\pm\sqrt{(\theta-1)^2-8}}{2},
\end{equation*}
and therefore $x_\pm=u_\pm^2$.

Thus, in the branch where the first spin value is distinguished, we obtain the
normalized fixed points $\frac{1}{x_-+2}(x_-,1,1)$ and $\frac{1}{x_++2}(x_+,1,1).$
The remaining fixed points are obtained by permuting the three spin values.
Together with the free fixed point $\left(\frac{1}{3},\frac{1}{3},\frac{1}{3}\right)$,
these exhaust all fixed points of $f$.

We now turn to the stability analysis of these fixed points on the simplex. For the symmetry-breaking fixed points, it is enough to compute the differential at $l=\frac{1}{u^2+2}
\left(u^2,1,1
\right)$ where $u\in \{\sqrt{x_+},\sqrt{x_-}\}$
because the remaining fixed points are obtained by permuting the spin values. Extending the map $f$ to an open subset of $\R^3$ containing $\Delta^2$, its ambient Jacobian in $\R^3$ is given by
\begin{equation}\label{eq: differential symmetry-breaking fixed point}
Df(l)
=\frac{1}{(u+1)(u^2+2)}
\begin{pmatrix}
4u & -2u^2 & -2u^2 \\
-2u & 2(u^2+1) & -2 \\
-2u & -2 & 2(u^2+1)
\end{pmatrix}.
\end{equation}
Here, we used the relation $\theta=\frac{u^2+u+2}{u}$, which follows from the fixed-point equation $f(l)=l$ for the symmetry-breaking fixed points.  We restrict this matrix to $T\Delta^2$.
A convenient basis of this tangent space is $v^{(1)}=(0,1,-1),$ and $v^{(2)}=(2,-1,-1)$.
From \eqref{eq: differential symmetry-breaking fixed point}, one obtains
\begin{equation*}
Df(l)v^{(1)}=\lambda_1(u)v^{(1)},~~\text{and}~~Df(l)v^{(2)}=\lambda_2(u)v^{(2)}
\end{equation*}
where $\lambda_1(u):=\frac{2}{u+1}$, and $\lambda_2(u):=\frac{2u(u+2)}{(u+1)(u^2+2)}$ are the eigenvalues of $Df(l)$ on the tangent space of the simplex.

We first consider $u_+=u_+(\theta)=\sqrt{x_+}$. Since $u_+$ is strictly increasing in $\theta$ and $u_+(1+2\sqrt{2})=\sqrt{2}$, it follows that $u_+>\sqrt{2}$ for all $\theta>1+2\sqrt{2}$. Hence, the first eigenvalue satisfies $\lambda_1(u_+)\in(0,1)$. Moreover, note that $(u+1)(u^2+2)-2u(u+2)=(u-1)(u^2-2)>0$ for every $u>\sqrt{2}$. Consequently, also the second eigenvalue satisfies $\lambda_2(u_+)\in (0,1)$ in the same parameter regime. Therefore, all coordinate permutations of $\frac{1}{x_++2}(x_+,1,1)$ are $f$-stable.

We next consider $u_-=u_-(\theta)=\sqrt{x_-}$. A straightforward computation shows that $u_-\in (0,1)$ for $\theta>4$, whereas $u_->1$ for $\theta\in (1+2\sqrt{2},4)$. Consequently, the first eigenvalue satisfies $\lambda_1(u_-)=\frac{2}{u_-+1}>1$ for $\theta>4$ and $\lambda_1(u_-)<1$ for $\theta>4$. On the other hand, the denominator of $\lambda_2(u_-)$ minus its numerator is given by $(u_--1)(u_-^2-2)$. For $\theta>4$, both factors are negative, so this difference is positive and hence $\lambda_2(u_-)<1$. Moreover, by Vieta's formula, the two solutions $u_\pm$ of $u^2-(\theta-1)u+2=0$ satisfy $u_-u_+=2$. Recall that $u_+>\sqrt{2}$ for $\theta\in(1+2\sqrt{2},4)$ and hence $u_-=\frac{2}{u_+}<\sqrt{2}$. Thus, for $\theta\in (1+2\sqrt{2},4)$, we have $1<u_-<\sqrt{2}$, so that $(u_--1)(u_-^2-2)<0$ and consequently $\lambda_2(u_-)>1$. To summarize, we have $\lambda_1(u_-)>1>\lambda_2(u_-)$ for $\theta>4$ whereas $\lambda_2(u_-)>1>\lambda_1(u_-)$ for $\theta\in(1+2\sqrt{2},4)$. Therefore, by permutation symmetry, the coordinate permutations of $\frac{1}{x_-+2}(x_-,1,1)$ are precisely the $f$-saddles. 

It remains to consider the free fixed point $l_{\text{free}}=(\frac{1}{3},\frac{1}{3},\frac{1}{3})$. At this fixed point the differential of $f$ at $l_{\text{free}}$ in an open subset of $\R^3$ reads 
\begin{equation*}
    Df(l_{\text{free}})=\frac{2(\theta-1)}{3(\theta+2)}\begin{pmatrix}
        2& -1& -1\\
        -1& 2& -1\\
        -1& -1& 2
    \end{pmatrix}.
\end{equation*}
For every $v\in T\Delta^2$, we have $ Df(l_{\text{free}})v=\lambda_{\text{free}}v$ where $\lambda_{\text{free}}:=\frac{2(\theta-1)}{(\theta+2)}$. Therefore, both eigenvalues of $Df(l_{\text{free}})$ on the tangent space $T\Delta^2$ are equal to $\lambda_{\text{free}}$. In the case $\theta>4$, we have $\lambda_{\text{free}}>1$, and the free fixed point has two unstable directions. In contrast, for $\theta\in(1+2\sqrt{2},4)$, we have $\lambda_{\text{free}}<1$, so the free fixed point is $f$-stable.
\end{proof}

\printbibliography

\end{document}